\documentclass[11pt]{amsart}
\usepackage[T1]{fontenc}
\usepackage[bigdelims,vvarbb]{newtxmath} 
\usepackage[cal=boondoxo]{mathalfa} 
\usepackage{a4wide}
\usepackage{mathtools,amsthm,amsfonts}
\usepackage{tikz-cd}
\usepackage{hyperref}
\hypersetup{
	colorlinks,%
	citecolor=black,%
	filecolor=black,%
	linkcolor=black,%
	urlcolor=black
}
\usepackage{newtxtext}
\usepackage{orcidlink}
\usepackage{graphicx}
\usepackage{comment}

\newcommand{\bigast}{\mathop{\scalebox{1.7}{$\ast$}}\displaylimits}

\theoremstyle{plain}
\newtheorem{theorem}{Theorem}[section]
\newtheorem{proposition}[theorem]{Proposition}
\newtheorem{lemma}[theorem]{Lemma}
\newtheorem{corollary}[theorem]{Corollary}
\numberwithin{equation}{section}

\newtheorem{maintheorem}{Theorem}

\theoremstyle{definition}
\newtheorem{definition}[theorem]{Definition}
\newtheorem{remark}[theorem]{Remark}
\newcommand{\cal}[1]{\mathcal{#1}}

\newcommand{\Sym}{\mathbb{S}}

\newcommand{\Scomp}{\mathsf{Scpx}}
\newcommand{\RelScomp}{\mathsf{RelScpx}}

\newcommand{\opi}[1]{\mathit{#1}}

\renewcommand*{\emptyset}{\varnothing}

\DeclareMathOperator{\id}{id}

\let\oldtocsection=\tocsection%
\let\oldtocsubsection=\tocsubsection%

\renewcommand{\tocsection}[2]{\hspace{0em}\rule{0pt}{14pt}\oldtocsection{#1}{#2}\bf}
\renewcommand{\tocsubsection}[2]{\hspace{2em}\oldtocsubsection{#1}{#2}}

\title{The relative join operad and polyhedral products}

\author{\orcidlinki{Mathieu Vallée}{0000-0001-6336-9225}}
\address{Universit\'e libre de Bruxelles (ULB), Brussels, Belgium}
\email{mathieu.vallee@protonmail.com}

\date{\today}

\begin{document}
\begin{abstract}
	An inclusion of nonvoid simplicial complexes induces an arrow of
	polyhedral products. Applying Ayzenberg's polyhedral join to both
	complexes gives a symmetric relative join operad. Its endpoint
	suboperads recover Abramyan--Panov substitution and Ayzenberg
	composition. Neither endpoint operad nor either relative join operad
	is finitely generated. The suboperad with nonvoid lower complex arises
	as the nonempty power-set quotient of the subset-inclusion operad.

	When tensoring preserves colimits of nonempty finite diagrams, the
	suboperad acts on arrows by polyhedral colimits up to coherent natural
	isomorphism. The action induces a set-operad algebra on isomorphism
	classes of arrows. The dual construction produces Stanley--Reisner
	quotient arrows, and the colimit action refines Eldridge's loop-space
	decomposition to arrows. PL ball--boundary pairs form a suboperad of
	the relative join operad, and minimal interior faces give an operad
	morphism. Principal pairs have odd-dimensional spherical moment-angle
	homotopy fibers and yield a closure result for Eldridge's loop-space
	class.
\end{abstract}
{\let\newpage\relax\maketitle}
\section*{Introduction}

Polyhedral products include moment-angle complexes and complements of
coordinate subspace arrangements
\cite{Bahri_Bendersky_Cohen_Gitler_2010,Buchstaber_Panov_2015}.
Ayzenberg's polyhedral join replaces each vertex of an outer complex by
a pair of inner complexes
\cite[Definition~4.2 and Observation~4.3]{Ayzenberg_2013}. Its
composition specialization forms an operad of nonvoid finite complexes
\cite[Proposition~4.12 and Remark~4.13]{Ayzenberg_2013}. Polyhedral
products over these composites satisfy nesting formulas
\cite[Proposition~5.1]{Ayzenberg_2013}, extended to general polyhedral
joins by Vidaurre~\cite[Theorem~2.9]{Vidaurre_2018}. These formulas
make operadic iteration visible in topology.

In their study of higher Whitehead products, Abramyan and Panov introduced
a substitution formula for ghost-free complexes
\cite[Definition~4.1 and Theorem~5.1]{Abramyan_Panov_2019}. We call the two
operations \emph{Ayzenberg composition} and \emph{Abramyan--Panov
substitution}. We use a \emph{fixed-ambient extension} of substitution
that retains the prescribed vertex set, including its ghost vertices.
Both operads exclude the void complex.

Inclusions supply the arrows studied here. Let $K$ be a
nonvoid simplicial complex on $[n]$ and let $f_i\colon Y_i\to X_i$. Each face
$I\in K$ selects $X_i$ for $i\in I$ and $Y_i$ otherwise; the colimit of
the resulting cubical diagram is $\cal Z_n(K;\underline f)$. For Cartesian
products and inclusions of CW subcomplexes, this is the usual polyhedral
product $\cal Z_K$~\cite[Definition~2.1]{Bahri_Bendersky_Cohen_Gitler_2010}. An inclusion of
nonvoid complexes $L\subseteq K$ then induces an arrow
\cite[Remark~2.3]{Bahri_Bendersky_Cohen_Gitler_2010}
\[
	\cal Z_n(L;\underline f)\longrightarrow\cal Z_n(K;\underline f).
\]
The nesting results above identify polyhedral products over composite
complexes objectwise. Retaining the induced arrow also exposes relative
information, such as kernels and homotopy fibers. Relative simplicial complexes
$L\subseteq K$ are standard, notably in piecewise-linear~(PL)
topology~\cite{Hudson_1969,Rourke_Sanderson_1982} and in relative
Stanley--Reisner theory
\cite{Adiprasito_Sanyal_2016}. We ask how the polyhedral join acts on
these inclusions and their associated arrows.

We first construct the operad of relative complexes and describe its
power-set origin. We then realize its suboperad with nonvoid lower complex
on arrows by polyhedral colimits and apply the resulting nesting identities
in algebra, PL topology, and homotopy theory.

For $n\geq1$, let $\RelScomp(n)$ consist of pairs $(K,L)$ of complexes
on $[n]$ with $L\subseteq K$. We retain this inclusion when applying
Ayzenberg's polyhedral join. If $K\triangleleft_k(M,N)$ denotes the
one-slot polyhedral join, set
\[
	(K,L)\circ_k^\ast(M,N)
	\coloneq
	\bigl(K\triangleleft_k(M,N),
	L\triangleleft_k(M,N)\bigr).
\]

\begin{maintheorem}[Relative join operad;
	Theorem~\ref{theorem:relative_operad},
	Proposition~\ref{proposition:subst_comp_subop}, and
	Theorem~\ref{theorem:inf_gen}]
	The operations $\circ_k^\ast$ make $\RelScomp$ a symmetric operad
	with unit $(\Delta_{[1]},\{\emptyset\})$. The relative complexes
	with nonvoid lower complex form a suboperad $\RelScomp_+$, and
	the fixed-ambient Abramyan--Panov substitution operad and Ayzenberg
	composition operad embed in it as the two endpoint suboperads. None of
	these four operads is finitely generated.
\end{maintheorem}

The endpoint embeddings send a nonvoid complex $K$ to
$(K,\{\emptyset\})$ and $(\Delta_{[n]},K)$, respectively. A common
arity-growth obstruction proves non-finite generation: a finitely
generated positive-arity set operad with finite unary part grows at most
exponentially up to the factor $n!$, whereas the number of simplicial
complexes grows double-exponentially. Relative join composition also
admits a link--deletion formula whose unit specialization is the classical
link--star--deletion pushout.

The power-set description reveals the operad underlying this construction.
Applying Giraudo's monoid-to-operad functor
\cite[Section~3.1]{Giraudo_2015} to the opposite endomorphism monoid of
the two-element set encodes its operations by upper-first pairs of subsets.
Restricting to the monotone endomorphisms of the chain $0<1$ gives the
subset-inclusion operad
\[
	\mathsf{Inc}(n)
	\coloneq\{(I_1,I_0)\mid I_0\subseteq I_1\subseteq[n]\}.
\]
These three endomorphisms correspond to the coordinate maps available to
an arrow.

Trimble observed that the covariant power-set functor is a monoidal
monad on species (symmetric collections) and that it sends set operads to
set operads
\cite[Sections~A.2--A.3]{Trimble_1999}.
Proposition~\ref{proposition:powerset_monads} establishes the induced
monad on set operads and its aritywise nonempty submonad
$\overline\wp$. The nonempty power-set operad
$\overline\wp(\opi{Perm})$ is canonically isomorphic to the basis
set operad $\opi{ComTrias}$ underlying Vallette's commutative
triassociative operad~\cite{Vallette_2007}.
We functorially recover $\RelScomp_+$ as follows.

\begin{maintheorem}[Power-set quotient;
	Theorem~\ref{theorem:relative_power_quotient}]
	Sending a nonempty family of subset inclusions to the downward closures
	of its two coordinate projections defines a surjection
	\[
		q\colon\overline\wp(\mathsf{Inc})
		\twoheadrightarrow\RelScomp_+.
	\]
	It is the unique extension of the simplex-pair map that respects
	nonempty unions, and it is compatible with operad composition. Two
	families have the same image precisely when the two resulting downward
	closures agree.
\end{maintheorem}

The two endpoint restrictions of $q$ recover
Abramyan--Panov substitution and Ayzenberg composition.

To realize the quotient on arrows, let $f\colon Y\to X$. The three
monotone coordinate maps act as
$\id_Y$, $f$, and $\id_X$. Their blockwise tensor products define an action
of $\mathsf{Inc}$ on the arrow category up to coherent natural isomorphism.
Passing from simplex pairs to unions glues these cubical arrows along common
subfaces and realizes the quotient $q$ categorically.

\begin{maintheorem}[Polyhedral-colimit action;
	Theorem~\ref{theorem:algebra_pair}]
	Let $(\mathrm C,\otimes,e)$ be a symmetric monoidal category admitting
	colimits of nonempty finite diagrams, preserved by tensoring with each
	object on either side. Every relative complex $(K,L)$ with nonvoid lower
	complex determines a polyhedral-colimit functor
	\[
		\cal Z_{K,L}\colon
		\operatorname{Arr}(\mathrm C)^n
		\longrightarrow\operatorname{Arr}(\mathrm C)
	\]
	compatible with relative join composition up to coherent natural
	isomorphism. On isomorphism classes these operations define an
	algebra over the set operad $\RelScomp_+$, sending $(K,L)$ and input arrows
	$\underline f$ to the class of
	$\cal Z_n(L;\underline f)\to\cal Z_n(K;\underline f)$.
\end{maintheorem}

If an initial object is present and preserved by tensoring, the action
extends to the full relative join operad. Under the dual hypotheses for
limits of nonempty finite diagrams, the same construction in
$\mathrm C^{\mathrm{op}}$ gives
polyhedral-limit operations up to coherent natural isomorphism.

The polyhedral-limit action yields Stanley--Reisner quotient arrows.
Relative join composition preserves PL ball--boundary pairs. Ayzenberg's
composition describes the resulting complexes, and Eldridge's theorem splits
their polyhedral-product loop spaces~\cite[Theorem~4.3]{Eldridge_2026}; for
insertion at a vertex of the lower complex, our relative refinement makes
these splittings compatible with the induced arrow and splits its looped
homotopy fiber. Minimal interior faces compose operadically; principal pairs
have odd-sphere moment-angle
homotopy fibers. For these pairs, Eldridge's loop-space class $\cal P$
\cite[Section~2.1]{Eldridge_2026} gives a further suboperad, with a
vertex-spanning criterion for its source loop spaces.

\section{Preliminaries}\label{section:recol}
We fix conventions for arrows and relative complexes, then recall the
operad constructions used below. General background can be found
in~\cite{Loday_Vallette_2012,MacLane_1998}.

\subsection{Complexes and arrows}\label{section:combinatorial_conventions}
We regard a poset as a category with one arrow $x\to y$ precisely when
$x\leq y$. For a category $\mathrm C$, write
$\operatorname{Ob}(\mathrm C)$ for its collection of objects. Its
\emph{arrow category} $\operatorname{Arr}(\mathrm C)$ has the arrows of
$\mathrm C$ as objects and commutative squares as morphisms. Write
\[
	\operatorname{IsoArr}(\mathrm C)
	\coloneq\operatorname{Ob}\operatorname{Arr}(\mathrm C)/\!\cong
\]
for its isomorphism classes, taken in a fixed universe. Limits and
colimits in $\operatorname{Arr}(\mathrm C)$, when they exist, are computed
at the source and target.

For a set $E$, let $\wp(E)$ be its power set and
$\overline{\wp}(E)\coloneq\wp(E)\setminus\{\emptyset\}$ its \emph{nonempty
	power set}. Thus $\wp(\emptyset)=\{\emptyset\}$ whereas
$\overline{\wp}(\emptyset)=\emptyset$.
Let $[n]=\{1,\ldots,n\}$, with $[0]=\emptyset$, and for integers $a,b$ let
$[a,b]=\{a,a+1,\ldots,b\}$, understood to be empty when $a>b$.
For $I\subseteq[n]$ and $k\in[n]$, put
\[
	I^{<k}=\{i\in I\mid i<k\},\qquad
	I^{>k}=\{i\in I\mid i>k\},\qquad
	I+\ell=\{i+\ell\mid i\in I\}\quad(\ell\geq0).
\]
These conventions describe the three blocks created when the $k$th
coordinate is replaced by an $m$-element block.
Disjoint union is denoted by $\sqcup$.
For $\sigma\in\Sym_n$, our right symmetric-group action is the relabeling
$I\cdot\sigma\coloneq\sigma^{-1}(I)$.
It extends elementwise to families of subsets and componentwise to
ordered pairs.

The \emph{Boolean lattice} is $\mathcal B_n\coloneq(\wp([n]),\subseteq)$.
Its objects are subsets of $[n]$ and its arrows are subset inclusions
$I_0\subseteq I_1$, also written as upper-first ordered pairs $(I_1,I_0)$.
For $n\geq1$, write
\[
	\mathsf{Inc}(n)
	\coloneq\operatorname{Ob}\operatorname{Arr}(\mathcal B_n)
	=\{(I_1,I_0)\mid I_0\subseteq I_1\subseteq[n]\}
\]
for the arity-indexed collection $\mathsf{Inc}$.

A \emph{simplicial complex} on $[n]$ is a lower set in
$\mathcal B_n$: a family $K\subseteq\wp([n])$ such that $I\in K$
and $J\subseteq I$ imply $J\in K$. Its elements are its
\emph{faces}. We allow the empty lower set: it is the \emph{void
complex} $\emptyset$. Every nonvoid complex contains the empty face;
$\{\emptyset\}$ is the \emph{trivial complex}.
Write $\mathcal D_n$ for the poset of these lower
sets, ordered by inclusion. For $n\geq1$, write
\[
	\Scomp(n)\coloneq\operatorname{Ob}(\mathcal D_n),\qquad
	\Scomp_{\neq\emptyset}(n)
	\coloneq\operatorname{Ob}(\mathcal D_n)\setminus\{\emptyset\}.
\]
These sets form the arity-indexed collections $\Scomp$ and
$\Scomp_{\neq\emptyset}$, respectively.

For $I\subseteq[n]$, let $\Delta_I\coloneq\{J\subseteq I\}$, regarded as a
complex on the fixed ambient vertex set $[n]$. A vertex $i$ is \emph{ghost}
in $K$ when $\{i\}\notin K$; thus the vertices outside $I$ are ghosts in
$\Delta_I$. Write $\Delta_{[n]}$ for the full simplex and, for every
$I\subseteq[n]$, put
\[
	\partial\Delta_I\coloneq\Delta_I\setminus\{I\}.
\]
In particular, $\Delta_\emptyset=\{\emptyset\}$ and
$\partial\Delta_{[1]}=\{\emptyset\}$.

The maximal faces of a simplicial complex are its \emph{facets}. For a
family $\mathcal A\subseteq\wp([n])$, its \emph{downward closure} is
\[
	\mathcal A^\downarrow
	\coloneq\bigcup_{I\in\mathcal A}\Delta_I
	=\{J\subseteq[n]\mid J\subseteq I\text{ for some }I\in\mathcal A\}.
\]
In particular, $K^\downarrow=K$ and
$\emptyset^\downarrow=\emptyset$.
Every simplicial complex is the downward closure of its facets.
Unions in $\mathcal D_n$ are ordinary unions of face sets; the empty
union is~$\emptyset$.
If \(K_1\) and \(K_2\) are simplicial complexes on disjoint ordered
vertex sets, their augmented simplicial join is
\[
	K_1\ast K_2\coloneq\{I\sqcup J\mid I\in K_1,\ J\in K_2\}.
\]
Its unit is \(\{\emptyset\}\), so an empty iterated join means
\(\{\emptyset\}\); the void complex is absorbing.
We order the disjoint union by concatenation in the displayed factor
order, with every vertex of the first factor preceding every vertex of
the second.
Throughout, ``join'' refers to this operation or to its polyhedral and
topological analogues; lattice suprema are called unions.

For a simplicial complex $K$ on $[n]$ and $k\in[n]$, define the link and
deletion at $k$ by
\[
	\operatorname{Lk}_k(K)\coloneq
	\{I\subseteq[n]\setminus\{k\}\mid I\cup\{k\}\in K\},
	\qquad
	K\setminus k\coloneq\{I\cap([n]\setminus\{k\})\mid I\in K\}.
\]
These complexes have ambient vertex set $[n]\setminus\{k\}$ unless
we explicitly retain a larger set.

A \emph{relative simplicial complex} on $[n]$ is an upper-first pair
$(K,L)$ of simplicial complexes on $[n]$
with $L\subseteq K$. We call $K$ the upper complex and $L$ the lower
complex. For $n\geq1$, write
\[
	\RelScomp(n)
	\coloneq\operatorname{Ob}\operatorname{Arr}(\mathcal D_n),\qquad
	\RelScomp_+(n)
	\coloneq\{(K,L)\in\RelScomp(n)\mid L\neq\emptyset\}.
\]
The identification $(L\subseteq K)\leftrightarrow(K,L)$ equips
$\RelScomp(n)$ with the componentwise order inherited from
$\operatorname{Arr}(\mathcal D_n)$. Its unions are computed componentwise:
\[
	\bigcup_\alpha(K_\alpha,L_\alpha)
	\coloneq
	\left(\bigcup_\alpha K_\alpha,\bigcup_\alpha L_\alpha\right).
\]

\subsection{Operad constructions}\label{section:operads}
We use symmetric, unital set operads with one-slot (partial) compositions
$a\circ_k b$~\cite[Chapter~5]{Loday_Vallette_2012}. For unital operads,
these determine total composition and keep formulas local to the
inserted block. All named combinatorial operads have positive arity
$n\geq1$ unless we explicitly allow arity zero.

For a set $X$, put
$\opi{End}_X(n)\coloneq\operatorname{Hom}_{\mathsf{Set}}(X^n,X)$; an
$\cal O$-algebra is an operad morphism
$\cal O\to\opi{End}_X$. Thus an algebra preserves the unit,
symmetric-group actions, and partial compositions.

For a monoid $(M,\mathbin\bullet,1_M)$, Giraudo's construction is the
positive-arity symmetric set operad
\[
	T(M)(n)\coloneq M^n
\]
with unit $(1_M)$ and partial compositions
\begin{equation}\label{equation:giraudo_composition}
	(x_1,\ldots,x_n)\circ_k(y_1,\ldots,y_m)
	\coloneq
	(x_1,\ldots,x_{k-1},
	x_k\mathbin\bullet y_1,\ldots,x_k\mathbin\bullet y_m,
	x_{k+1},\ldots,x_n).
\end{equation}
Its right symmetric-group action permutes coordinates:
$(x\cdot\sigma)_r\coloneq x_{\sigma(r)}$.
A monoid homomorphism acts coordinatewise, so $T$ sends submonoids to
suboperads. The construction, operad axioms, and functoriality are
given in~\cite[Section~3.1, Equations~(3.4)--(3.6), Proposition~3.1,
and Theorem~3.4]{Giraudo_2015}.

The power-set construction applies independently to positive-arity operads
and to operads allowing arity zero. For a symmetric unital set operad
$\cal O$, define in every permitted arity
\[
	(\wp\cal O)(n)\coloneq\wp(\cal O(n)),\qquad
	(\overline\wp\cal O)(n)\coloneq\overline\wp(\cal O(n)).
\]
Trimble observed the full covariant monoidal power-set monad
in~\cite[Sections~A.2--A.3]{Trimble_1999}. The following proposition
gives the induced monad on set operads and its aritywise nonempty submonad.

\begin{proposition}[Power-set monads]\label{proposition:powerset_monads}
	Equip these collections with the actions
	$\mathcal A\cdot\sigma\coloneq\{a\cdot\sigma\mid a\in\mathcal A\}$ and with
	the formulas
	\[
		\mathcal A\circ_k^\wp\mathcal B
		\coloneq\{a\circ_k b\mid a\in\mathcal A,\ b\in\mathcal B\},
		\qquad
		1_{\wp\mathcal O}\coloneq\{1_{\mathcal O}\},\qquad
		1_{\overline\wp\mathcal O}\coloneq\{1_{\mathcal O}\}.
	\]
	Then $\wp\mathcal O$ and $\overline\wp\mathcal O$ are symmetric set
	operads, and direct image under operad morphisms makes $\wp$ and
	$\overline\wp$ endofunctors both on positive-arity symmetric unital set
	operads and on symmetric unital set operads allowing arity zero. The maps
	\[
		\eta_{\mathcal O}^\wp\colon\mathcal O\longrightarrow\wp\mathcal O,
		\quad a\longmapsto\{a\},
		\qquad
		\mu_{\mathcal O}^\wp\colon\wp^2\mathcal O\longrightarrow\wp\mathcal O,
		\quad \mathfrak A\longmapsto
		\bigcup_{\mathcal A\in\mathfrak A}\mathcal A
	\]
	are the unit and multiplication of a monad structure on $\wp$, and
	they restrict to a submonad structure on $\overline\wp$.
\end{proposition}
\begin{proof}
	Each partial-composition and equivariance identity follows by applying
	the corresponding identity in $\mathcal O$ to every choice of elements.
	Direct image under an operad morphism preserves these formulas, so
	$\wp$ is functorial.
	The singleton map preserves the unit and symmetric actions, and it
	preserves partial compositions because
	$\{a\}\circ_k^\wp\{b\}=\{a\circ_k b\}$.

	For families of subsets $\mathfrak A$ and $\mathfrak B$, union is an
	operad morphism since
	\[
		\mu^\wp(\mathfrak A\circ_k^{\wp\wp}\mathfrak B)
		=\bigcup_{\substack{\mathcal A\in\mathfrak A\\
		                      \mathcal B\in\mathfrak B}}
		  (\mathcal A\circ_k^\wp\mathcal B)
		=\left(\bigcup_{\mathcal A\in\mathfrak A}\mathcal A\right)
		  \circ_k^\wp
		  \left(\bigcup_{\mathcal B\in\mathfrak B}\mathcal B\right).
	\]
	It also preserves the unit and symmetric actions. Direct image commutes
	with singleton formation and unions, so $\eta^\wp$ and $\mu^\wp$ are
	natural operad morphisms. Aritywise, the two unit laws say that taking
	the union of a singleton family, or of the family of singleton subsets,
	changes nothing; associativity says that either order of flattening a
	family of families of subsets gives the same union. These are the monad laws.
	Finally, direct images, singleton maps, pointwise partial compositions, and
	unions of nonempty families of nonempty subsets preserve nonemptiness.
	Thus the aritywise inclusions
	$\overline\wp\mathcal O\hookrightarrow\wp\mathcal O$ commute with the
	units and multiplications, and exhibit $\overline\wp$ as a submonad.
\end{proof}
We call $\wp\mathcal O$ and $\overline\wp\mathcal O$ the \emph{full}
and \emph{nonempty power-set operads} of $\mathcal O$. The singleton map
embeds $\mathcal O$ into $\wp\mathcal O$. The empty subset in each arity
is fixed by the symmetric action, absorbs every partial composition, and
does not contain the operadic unit.

For the named power-set operads below, we suppress the superscript once
their partial compositions have been identified. We retain
$\circ_k^\wp$ for the generic power-set partial composition and for
powers of $\mathsf{Inc}$.

The permutative operad gives a familiar example.
Let $\opi{Perm}$ be the set operad with $\opi{Perm}(n)\coloneq[n]$ for $n\geq1$
(and $\opi{Perm}(0)\coloneq\emptyset$ in the version allowing arity zero). Its
right symmetric-group action is $i\cdot\sigma\coloneq\sigma^{-1}(i)$. For $i\in[n]$,
$j\in[m]$, and $k\in[n]$, its partial compositions are
	\[
		i\circ_k j\coloneq
		\begin{cases}
			i+m-1, & k<i, \\
			i+j-1, & k=i, \\
			i,     & k>i.
		\end{cases}
	\]
Its linearization is Chapoton's permutative operad~\cite{Chapoton_2001}.
For nonempty subsets $I\subseteq[n]$, $J\subseteq[m]$, and $k\in[n]$,
the nonempty power-set construction gives
	\begin{equation}\label{equation:comtrias_composition}
		I\circ_k J=
		\begin{cases}
			I^{<k}\sqcup(J+(k-1))\sqcup(I^{>k}+(m-1)), & k\in I,    \\
			I^{<k}\sqcup(I^{>k}+(m-1)),                 & k\notin I.
		\end{cases}
	\end{equation}
Vallette's linear operad $\mathcal{ComTrias}$ has a composition-stable
basis indexed by nonempty subsets~\cite[Theorem~14 and
Corollary~15]{Vallette_2007}. Denoting its basis set operad by
$\opi{ComTrias}$, the identity maps give
	\[
		\overline\wp(\opi{Perm})\cong\opi{ComTrias}.
	\]

\section{The relative join operad}
\label{section:operad_pairs}
One-slot polyhedral joins define an operad of relative complexes that
contains two classical operads. We also examine generation, links, and
deletions.

\subsection{Relative join composition}
\label{section:polyhedral_join_prod}
\label{section:relative_composition}
\begin{definition}[Polyhedral join]
	\label{definition:simplicial_join_product}
	Let \(K\) be a simplicial complex on \([n]\) and let
	\((M_i,N_i)\) be relative complexes on pairwise disjoint
	ordered vertex blocks. For $I\subseteq[n]$, put
	\[
		P_i(I)=
		\begin{cases}
			M_i, & i\in I,\\
			N_i, & i\notin I.
		\end{cases}
	\]
	Their \emph{polyhedral join over \(K\)} is
	\[
		\mathcal J_K(\underline M,\underline N)
		\coloneq\bigcup_{I\in K}\bigast_{i=1}^nP_i(I),
	\]
	where the join factors occur in increasing slot order.
\end{definition}
This is Ayzenberg's polyhedral join~\cite[Definition~4.2]{Ayzenberg_2013},
and the union is a simplicial complex~\cite[Observation~4.3]{Ayzenberg_2013}.
It is a subcomplex of \(\bigast_{i=1}^nM_i\) and, when
\(K\neq\emptyset\), contains \(\bigast_{i=1}^nN_i\).
It respects inclusions: if $K\subseteq K'$ and, for every $i$,
$M_i\subseteq M_i'$ and $N_i\subseteq N_i'$, then
\[
	\mathcal J_K(\underline M,\underline N)
	\subseteq\mathcal J_{K'}(\underline M',\underline N').
\]

To isolate one slot, let $R$ be a complex on $[n]$, let $(M,N)$ be a
relative complex on $[m]$, and choose $k\in[n]$. Define a complex on
$[n+m-1]$ by
\begin{equation}\label{equation:relative_face_action}
	\begin{aligned}
		R\triangleleft_k(M,N)
		\coloneq{} & \{I^{<k}\sqcup(J_M+(k-1))\sqcup(I^{>k}+(m-1))
		\mid I\in R,\ k\in I,\ J_M\in M\}             \\
		    & {}\cup\{I^{<k}\sqcup(J_N+(k-1))\sqcup(I^{>k}+(m-1))
		\mid I\in R,\ k\notin I,\ J_N\in N\}.
	\end{aligned}
\end{equation}

The polyhedral join with $(M,N)$ in slot $k$ and
$(\Delta_{[1]},\{\emptyset\})$ elsewhere has exactly the faces in
Equation~\eqref{equation:relative_face_action}; downward closure supplies
the outside subfaces.
Hence the result is a simplicial complex and is monotone in every input.
Applying the face formula to both members of a relative complex defines
\emph{relative join composition}
\[
	(K,L)\circ_k^\ast(M,N)\coloneq
	\bigl(K\triangleleft_k(M,N),L\triangleleft_k(M,N)\bigr).
\]

\begin{theorem}[The relative join operad]
	\label{theorem:relative_operad}
	The partial compositions $\circ_k^\ast$ and the componentwise right
	relabeling action make $\RelScomp$ a symmetric operad with unit
	$(\Delta_{[1]},\{\emptyset\})$.
	The collection $\RelScomp_+$ is a unital suboperad.
\end{theorem}
\begin{proof}
	Monotonicity makes the displayed pair relative. For a complex $R$
	and a relative complex $(M,N)$, the unit identities are
	\[
		\Delta_{[1]}\triangleleft_1(M,N)=M,
		\qquad
		\{\emptyset\}\triangleleft_1(M,N)=N,
		\qquad
		R\triangleleft_k(\Delta_{[1]},\{\emptyset\})=R.
	\]
	The one-slot face formula gives nested associativity
	\[
		\bigl(R\triangleleft_k(M,N)\bigr)
		\triangleleft_{k+j-1}(P,Q)
		=R\triangleleft_k
		\bigl(M\triangleleft_j(P,Q),
		N\triangleleft_j(P,Q)\bigr),
	\]
	and, for $k<h$, disjoint associativity
	\[
		\bigl(R\triangleleft_k(M,N)\bigr)
		\triangleleft_{h+m-1}(P,Q)
		=
		\bigl(R\triangleleft_h(P,Q)\bigr)
		\triangleleft_k(M,N).
	\]
	Each identity makes the same successive face-membership choices in
	the same ordered blocks. Exchanging the disjoint slots gives the
	reversed order, and componentwise relabeling preserves these choices.
	Applying the set equalities to the upper and lower complexes proves
	the operad axioms.
	The unit belongs to $\RelScomp_+$; when $L$ and $N$ are nonvoid,
	$\emptyset\in L\triangleleft_k(M,N)$, so $\RelScomp_+$ is a
	suboperad.
\end{proof}

We call $\RelScomp$ with these operations the \emph{relative join operad}.
Iterating insertions from right to left in the original slots recovers
$\mathcal J_K(\underline M,\underline N)$ and
$\mathcal J_L(\underline M,\underline N)$ as the upper and lower
components, respectively, with their canonical block order.

\subsection{Endpoint suboperads}
\label{section:endpoint_operads}
Fixing the lower complex to be trivial, or the upper complex to be
the full simplex, gives Abramyan--Panov substitution and Ayzenberg
composition, respectively. For nonvoid complexes $K$ on $[n]$ and $L$ on
$[m]$, set
\[
	K\circ_k L\coloneq K\triangleleft_k(L,\{\emptyset\}),
	\qquad
	K\circ_k^c L\coloneq K\triangleleft_k(\Delta_{[m]},L).
\]
\begin{proposition}[Endpoint suboperads]
	\label{proposition:subst_comp_subop}
	The operations $\circ_k$ make $\Scomp_{\neq\emptyset}$ a symmetric
	operad with unit $\Delta_{[1]}$. On ghost-free complexes they agree
	with Abramyan--Panov substitution~\cite[Definition~4.1]{Abramyan_Panov_2019},
	and the same formulas
	extend substitution to prescribed ambient vertex sets with ghost vertices.
	The operations $\circ_k^c$ make
	$\Scomp_{\neq\emptyset}$ Ayzenberg's composition operad, with unit
	$\{\emptyset\}$~\cite[Definition~4.4, Proposition~4.12, and
	Remark~4.13]{Ayzenberg_2013}.

	The assignments
	\[
		\iota_{\mathrm{sub}}(K)\coloneq(K,\{\emptyset\}),\qquad
		\iota_{\mathrm{comp}}(K)\coloneq(\Delta_{[n]},K)
	\]
	embed the Abramyan--Panov substitution and Ayzenberg composition operads,
	respectively, as suboperads of $\RelScomp_+$, hence of $\RelScomp$.
\end{proposition}
\begin{proof}
	Equation~\eqref{equation:relative_face_action} gives the stated partial
	compositions. Under relative join composition, the first embedding
	retains the trivial lower complex $\{\emptyset\}$, while the second
	retains the full simplex as its upper complex. Both respect componentwise
	relabeling, and both endpoint units map to
	$(\Delta_{[1]},\{\emptyset\})$.
\end{proof}

\subsection{Non-finite generation}
One counting bound rules out finite generation for all four operads.

\begin{lemma}[Growth of finitely generated set operads]
	\label{lemma:finitely_generated_growth}
	Let $\cal O$ be a positive-arity symmetric set operad with
	$\cal O(1)$ finite. If $\cal O$ is finitely generated, then
	\[
		|\cal O(n)|\leq C^n n!
	\]
	for some constant $C$ and every $n\geq1$.
\end{lemma}

\begin{proof}
	Positive arity rules out nullary vertices in composite trees. Choose
	a finite generating set $G$ and close it under the symmetric-group
	actions; the result is still finite. Let $d$ be the
	maximum of $2$ and the arities in $G$. Every arity-$n$ composite has
	a planar rooted tree representation with $G$-labeled vertices and
	$n$ labeled leaves. Contract each maximal chain of unary vertices.
	The composite along that chain, an element of the finite set
	$\cal O(1)$, decorates the corresponding edge of the reduced tree;
	use the unit for an empty chain.

	If the reduced tree has $v$ nonunary vertices, of arities between $2$
	and $d$, then
	$\sum_x(\operatorname{arity}(x)-1)=n-1$, so $v\leq n-1$. It has
	$n+v\leq2n-1$ edges when the external output edge is included. The
	number of shapes with $e$ edges is at most the Catalan number
	$\operatorname{Cat}_e\leq4^e$, so the number of shapes with at most
	$2n-1$ edges is exponentially bounded. Vertex labels and edge
	decorations contribute another exponential factor, while leaf labels
	contribute $n!$.
	Increasing the exponential constant to cover $n=1$ gives the bound.
\end{proof}

\begin{theorem}[Non-finite generation]
	\label{proposition:relative_inf_gen}
	\label{theorem:inf_gen}
	The relative join operads $\RelScomp$ and $\RelScomp_+$ and the two
	endpoint operads
	\[
		(\Scomp_{\neq\emptyset},\{\circ_k\}),\qquad
		(\Scomp_{\neq\emptyset},\{\circ_k^c\})
	\]
	are not finitely generated.
\end{theorem}
\begin{proof}
	All four unary sets are finite. For $n\geq2$, put
	$r=\lfloor n/2\rfloor$. For each
	$\mathcal A\subseteq\binom{[n]}{r}$, set
	\[
		K_{\mathcal A}
		=\{\emptyset\}\cup\bigcup_{F\in\mathcal A}\Delta_F.
	\]
	The $r$-element faces of $K_{\mathcal A}$ are exactly the members of
	$\mathcal A$. Thus the $K_{\mathcal A}$ give distinct operations in
	each endpoint operad, and the pairs
	$(K_{\mathcal A},\{\emptyset\})$ give distinct operations in each
	relative join operad. Each of the four arity-$n$ sets therefore has
	cardinality at least
	\[
		2^{\binom{n}{r}}.
	\]
	Since $\binom{n}{r}\geq2^n/(n+1)$, whereas
	$\log_2(C^n n!)=O(n\log n)$, this contradicts
	Lemma~\ref{lemma:finitely_generated_growth} if any of the four
	operads is finitely generated.
\end{proof}

\subsection{Link--deletion formulas}\label{section:link_deletion}
We regard $\operatorname{Lk}_k(K)$ and $K\setminus k$ as complexes on
$[n]$ with $k$ retained as a ghost vertex.

\begin{definition}[Inserted join]\label{definition:inserted_join}
	For complexes $P$ on $[n]$ and $Q$ on $[m]$, with $k$ a ghost
	vertex of $P$, define
	\[
		P\ast_k Q\coloneq
		\bigl\{I^{<k}\sqcup(J+(k-1))\sqcup(I^{>k}+(m-1))
		\mid I\in P,\ J\in Q\bigr\}.
	\]
	This is the ordinary augmented join of $P$, viewed on
	$[n]\setminus\{k\}$, and $Q$, with the $Q$-block placed in slot $k$
	by the canonical relabeling. The subscript records the output labels.
	We use $\ast_k$ for identities on standard labeled ambient vertex
	sets and $\ast$ when labels are immaterial.
\end{definition}

Splitting outer faces according to whether they contain $k$ gives a
link--deletion formula for the one-slot operation.
\begin{proposition}\label{proposition:join_comp_formula}
	Let $K$ be a simplicial complex on $[n]$ and $(M,N)$ a relative
	complex on $[m]$.
	For every $k\in[n]$:
	\[
		K\triangleleft_k(M,N)
		= (\operatorname{Lk}_k(K)\ast_k M)\cup((K\setminus k)\ast_k N).
	\]
\end{proposition}
\begin{proof}
	By downward closure,
	$K\setminus k=\{I\in K\mid k\notin I\}$. In
	Equation~\eqref{equation:relative_face_action}, removing $k$ from
	the outer faces that contain it gives $\operatorname{Lk}_k(K)$.
	The remaining outer faces form $K\setminus k$. Their inserted
	joins use $M$ and $N$, respectively.
\end{proof}
Since $\operatorname{Lk}_k(K)\subseteq K\setminus k$ and $N\subseteq M$,
and since each face splits uniquely into outside and inserted blocks,
the two terms intersect in $\operatorname{Lk}_k(K)\ast_kN$. Their
union is therefore a pushout of simplicial complexes. With
$\operatorname{st}_k(K)\coloneq\operatorname{Lk}_k(K)\ast_k\Delta_{[1]}$,
the choice $(M,N)=(\Delta_{[1]},\{\emptyset\})$ recovers the classical
link--star--deletion pushout
\[
	K=\operatorname{st}_k(K)\cup(K\setminus k),\qquad
	\operatorname{st}_k(K)\cap(K\setminus k)=\operatorname{Lk}_k(K).
\]

Two arity-one specializations recover link and deletion:
\[
	\operatorname{Lk}_k(K)
	=K\triangleleft_k(\{\emptyset\},\emptyset),\qquad
	K\setminus k
	=K\triangleleft_k(\{\emptyset\},\{\emptyset\}).
\]
\begin{remark}
	The link specialization uses the void lower complex in
	$(\{\emptyset\},\emptyset)\in\RelScomp(1)$. The deletion
	specialization uses
	$(\{\emptyset\},\{\emptyset\})\in\RelScomp_+(1)$, whose two
	complexes are nonvoid. In particular, the link formula also gives
	$\operatorname{Lk}_k(K)=\emptyset$ when $k$ is a ghost vertex.
\end{remark}
For a nonvoid complex $K$, the boundary specialization is
\[
	K\triangleleft_k(\Delta_{[2]},\partial\Delta_{[2]})
	=K\circ_k^c\partial\Delta_{[2]}.
\]
For a non-ghost vertex $k$ of $K$, this is the simplicial wedge of
Provan and Billera~\cite[p.~578]{Provan_Billera_1980}, also called the
\emph{one-vertex suspension}. The same formula extends it to a ghost
vertex in a fixed ambient set, creating a regular suspension.

The face formula also gives
$\partial\Delta_{[n]}\circ_k^c\partial\Delta_{[m]}
=\partial\Delta_{[n+m-1]}$; thus simplex boundaries form a unital copy
of the terminal positive-arity set operad $\opi{Com}$,
which has one operation in every positive arity.
In the ghost-free convention,
$K\circ_k^c\partial\Delta_{[m]}$ is the $K(J)$-construction for
$J=(j_1,\ldots,j_n)$, where $j_k=m$ and $j_i=1$ for $i\neq k$
\cite[Definition~2.1]{Bahri_Bendersky_Cohen_Gitler_2015}
\cite[Example~4.11]{Ayzenberg_2013}.
Equivalently, Ayzenberg's composition operad is a right
$\opi{Com}$-module via its simplex-boundary suboperad, and this action
gives the $K(J)$-construction.

\section{Power-set presentation}
\label{section:inclusion_framework}
Monotone self-maps of a two-element set encode subset inclusions.
Nonempty families of these inclusions present $\RelScomp_+$ through
downward closure.

\subsection{Subset-inclusion operad}
\label{section:inclusion_operad}
Let $\mathbf2\coloneq\{0,1\}$, ordered by $0<1$, and put
\[
	E\coloneq\operatorname{End}_{\mathsf{Set}}(\mathbf2)^{\mathrm{op}}
	=\{c_0,\id,\tau,c_1\},\qquad
	E_{\leq}\coloneq\operatorname{End}_{\mathsf{Poset}}(\mathbf2)^{\mathrm{op}}
	=\{c_0,\id,c_1\}.
\]
Here $c_s$ is constant with value $s$, $\tau$ exchanges $0$ and $1$,
and multiplication is $f\mathbin\bullet g\coloneq g\circ f$.
For a word $f=(f_1,\ldots,f_n)\in T(E)(n)$, put
\[
	I_s(f)\coloneq\{r\in[n]\mid f_r(s)=1\}\qquad(s=0,1)
\]
and record it by the upper-first pair $(I_1(f),I_0(f))$. Writing each
coordinate map as the word $f_i(0)f_i(1)$,
the dictionary is
\[
	00\longleftrightarrow c_0,\qquad
	01\longleftrightarrow\id,\qquad
	10\longleftrightarrow\tau,\qquad
	11\longleftrightarrow c_1.
\]
The monotone words are precisely those without the state $10$.
\begin{definition}[Subset-inclusion operad]
	\label{definition:subset_inclusion_operad}
	Transporting Giraudo's operad structure across the upper-first dictionary
	$T(E_{\leq})\cong\mathsf{Inc}$ defines the \emph{subset-inclusion operad}.
	Its unit is
	$(\{1\},\emptyset)$, and its symmetric action relabels both subsets
	componentwise.
\end{definition}

Equation~\eqref{equation:giraudo_composition} determines the partial
composition as follows.
For $(I_1,I_0)\in\mathsf{Inc}(n)$,
$(J_1,J_0)\in\mathsf{Inc}(m)$, and $k\in[n]$, we have
\begin{equation}\label{equation:inc_composition}
	(I_1,I_0)\circ_k(J_1,J_0)
	=(I_1\triangleleft_k(J_1,J_0),I_0\triangleleft_k(J_1,J_0)),
\end{equation}
where, for $s=0,1$,
\[
	I_s\triangleleft_k(J_1,J_0)=\begin{cases}
		I_s^{<k}\sqcup(J_1+(k-1))\sqcup(I_s^{>k}+(m-1)),& k\in I_s,\\
		I_s^{<k}\sqcup(J_0+(k-1))\sqcup(I_s^{>k}+(m-1)),& k\notin I_s.
	\end{cases}
\]

The submonoids $E_0\coloneq\{c_0,\id\}$ and $E_1\coloneq\{\id,c_1\}$ consist of the
maps preserving $0$ and $1$, respectively. Functoriality of $T$ gives the
endpoint suboperads of $\mathsf{Inc}$, with arity-$n$ dictionaries
\[
	\begin{aligned}
		\mathsf{Inc}_0\coloneq T(E_0)&:\quad
		I\longmapsto(I,\emptyset)\quad
		\text{or, equivalently, }\emptyset\subseteq I,\\
		\mathsf{Inc}_1\coloneq T(E_1)&:\quad
		I\longmapsto([n],I)\quad
		\text{or, equivalently, }I\subseteq[n].
	\end{aligned}
\]
Both have underlying collection $\{\wp([n])\}_{n\geq1}$. The partial
composition in $\mathsf{Inc}_0$ is the right-hand side of
Equation~\eqref{equation:comtrias_composition}, now for arbitrary subsets,
and its unit is $\{1\}$.

Conjugation by the order-reversing bijection of $\mathbf2$ gives the operadic
involution
\[
	(I_0\subseteq I_1)\longmapsto
	([n]\setminus I_1\subseteq[n]\setminus I_0).
\]
It exchanges $\mathsf{Inc}_0$ and $\mathsf{Inc}_1$, so the latter has unit
$\emptyset$ and its compositions are transported by complementation.
Although $\mathsf{Inc}_0$ and $\wp(\opi{Perm})$ have the same underlying
collection, their compositions differ on empty-subset inputs. The common
suboperad indexed by nonempty subsets is
$\overline\wp(\opi{Perm})\cong\opi{ComTrias}$; its complementary copy in
$\mathsf{Inc}_1$ is indexed by proper subsets of $[n]$.

\subsection{Downward-closure quotient}
\label{section:principal_pair_quotient}
Each subset inclusion determines a pair of simplices, giving the
\emph{simplex-pair map}
\[
	p_n\colon
	\mathsf{Inc}(n)\longrightarrow\RelScomp_+(n),\qquad
	(I_1,I_0)\longmapsto(\Delta_{I_1},\Delta_{I_0}).
\]
Recall that $\Delta_\emptyset=\{\emptyset\}$.
Since every $\Delta_I$ is nonvoid, $p$ lands in $\RelScomp_+$.
Here an \emph{algebra over a power-set monad} is an Eilenberg--Moore
algebra on positive-arity symmetric set operads.

\begin{lemma}[Simplex pairs and componentwise unions]
	\label{lemma:principal_pairs_unions}
	The aritywise map $p\colon\mathsf{Inc}\to\RelScomp_+$ is an operad
	morphism. Moreover, componentwise union
	\[
		\begin{aligned}
		\alpha_n\colon\overline\wp(\RelScomp_+(n))
			&\longrightarrow\RelScomp_+(n),\\
		\mathfrak K&\longmapsto
		\left(\bigcup_{(K,L)\in\mathfrak K}K,
		      \bigcup_{(K,L)\in\mathfrak K}L\right)
		\end{aligned}
	\]
	makes $\RelScomp_+$ an algebra over the $\overline\wp$-monad on
		positive-arity symmetric set operads.
\end{lemma}
\begin{proof}
	It preserves the unit because $p_1(\{1\},\emptyset)
		=(\Delta_{\{1\}},\Delta_\emptyset)$,
	and it commutes with relabeling since
	$\Delta_{\sigma^{-1}(I)}=\Delta_I\cdot\sigma$.
	For $(I_1,I_0)\in\mathsf{Inc}(n)$,
	$(J_1,J_0)\in\mathsf{Inc}(m)$, and $k\in[n]$, write
	$(H_1,H_0)=(I_1,I_0)\circ_k(J_1,J_0)$. Then, for $s=0,1$, we have
	$\Delta_{I_s}\triangleleft_k(\Delta_{J_1},\Delta_{J_0})
		=\Delta_{H_s}$.
	In fact, if $k\in I_s$, Equation~\eqref{equation:inc_composition} produces the simplex with inserted
	block $J_1$. Otherwise the block is $J_0$. This is the $s$th
	component $H_s$ of the block dictionary. Thus $p$ preserves partial
	compositions and is an operad morphism.

	Nonempty componentwise unions preserve $\RelScomp_+$. For nonempty indexed
	families, the face formula gives
	\[
		\left(\bigcup_aR_a\right)\triangleleft_k
		\left(\bigcup_bM_b,\bigcup_bN_b\right)
		=\bigcup_{a,b}\bigl(R_a\triangleleft_k(M_b,N_b)\bigr).
	\]
	Applying this identity to the upper and lower complexes shows that
	$\circ_k^\ast$ distributes over nonempty componentwise unions in each
	variable. Relabeling commutes with unions, and the singleton family
	containing the relative unit unions to that unit. Thus $\alpha$
	is an operad morphism. Moreover,
	$\alpha\circ\eta^\wp$ is the identity, while $\alpha\circ\overline\wp(\alpha)=\alpha\circ\mu^\wp$ because both sides flatten a nonempty family of nonempty families by
	componentwise union. Hence $\alpha$ satisfies the $\overline\wp$-algebra laws.
\end{proof}

The quotient is described by coordinatewise downward closure. For any
family $\mathcal A\subseteq\mathsf{Inc}(n)$, let
\[
	\mathcal A_s=\{I_s:(I_1,I_0)\in\mathcal A\}
	\qquad(s=0,1)
\]
be its two coordinate projections. When $\mathcal A$ is nonempty,
define
\begin{equation}\label{equation:q_downward_closure}
	q_n(\mathcal A)
	\coloneq\bigl(\mathcal A_1^\downarrow,
	\mathcal A_0^\downarrow\bigr).
\end{equation}
Both complexes are nonvoid, and
$\mathcal A_0^\downarrow\subseteq\mathcal A_1^\downarrow$ because
every member of $\mathcal A$ is a subset inclusion. Thus $q_n$ takes values
in $\RelScomp_+(n)$.

The symbol $\emptyset$ plays three distinct roles here: the empty
family in $\wp(\opi{Perm})(n)$, the empty-subset operation in
$\mathsf{Inc}_0(n)$, and the void complex in $\Scomp(n)$.
In particular,
\[
	\{\emptyset\}^{\downarrow}=\{\emptyset\},
	\qquad
	\emptyset^{\downarrow}=\emptyset.
\]
The singleton family containing the empty-subset operation of
$\mathsf{Inc}_0$ yields the trivial complex. Under the full-power-set
extension below, only the empty family yields the void complex, which
belongs to neither simplicial endpoint operad.

\begin{theorem}[The power-set quotient]
	\label{theorem:relative_power_quotient}
	In positive arities, the unique $\overline\wp$-algebra morphism from
	the free $\overline\wp$-algebra on $\mathsf{Inc}$ that extends $p$ is
	the surjection
	\[
	q\colon\overline\wp(\mathsf{Inc})
		\twoheadrightarrow\RelScomp_+
	\]
	given by Equation~\eqref{equation:q_downward_closure}. Equivalently,
	$q=\alpha\circ\overline\wp(p)$.
	Its kernel congruence is
	\[
		\mathcal A\sim\mathcal B
		\quad\Longleftrightarrow\quad
		\mathcal A_0^\downarrow=\mathcal B_0^\downarrow
		\qquad\text{and}\qquad
		\mathcal A_1^\downarrow=\mathcal B_1^\downarrow.
	\]
	Consequently,
	\[
		\overline\wp(\mathsf{Inc})/{\sim}
		\ \cong\ \RelScomp_+
	\]
	as $\overline\wp$-algebras, hence as operads.
\end{theorem}
\begin{proof}
	For a nonempty family $\mathcal A$, componentwise union gives
	\[
		\alpha_n\bigl(\overline\wp(p_n)(\mathcal A)\bigr)
		=\left(\bigcup_{(I_1,I_0)\in\mathcal A}\Delta_{I_1},
		        \bigcup_{(I_1,I_0)\in\mathcal A}\Delta_{I_0}
		 \right)
		=(\mathcal A_1^\downarrow,\mathcal A_0^\downarrow).
	\]
	Thus $q=\alpha\circ\overline\wp(p)$. By the free-algebra property
	of the monad, it is the unique
	$\overline\wp$-algebra morphism satisfying
	$q\circ\eta^\wp=p$; in particular, it is an operad morphism.

	For $(K,L)\in\RelScomp_+(n)$, let
	\[
		\mathcal A_{K,L}
		=\{(F,\emptyset)\mid F\text{ is a facet of }K\}
		\cup\{(G,G)\mid G\text{ is a facet of }L\}
		\subseteq\mathsf{Inc}(n).
	\]
	Since $L$ is nonvoid, it has a facet and the family is nonempty.
	Its upper projection generates $K$ under downward closure because
	$L\subseteq K$; its lower projection generates $L$. Thus
	$q_n(\mathcal A_{K,L})=(K,L)$. This proves surjectivity.

	Equation~\eqref{equation:q_downward_closure} shows that two
	families have the same image exactly when their coordinatewise
	downward closures agree. This identifies the kernel congruence and
	gives the quotient isomorphism.
\end{proof}

\begin{remark}
	Componentwise union also makes $\RelScomp$ a $\wp$-algebra; its empty
	union is the absorbing relative complex $(\emptyset,\emptyset)$.
	Allowing empty unions therefore gives the $\wp$-algebra morphism
	\[
		\widetilde q\colon\wp(\mathsf{Inc})\longrightarrow\RelScomp,
		\qquad
		\widetilde q_n(\mathcal A)
		\coloneq(\mathcal A_1^\downarrow,\mathcal A_0^\downarrow),
	\]
	whose image is
	$\RelScomp_+\cup\{(\emptyset,\emptyset)\}$. Applying the nonempty
	power-set construction to the endpoint suboperads and restricting $q$
	gives surjections
	$\overline\wp(\mathsf{Inc}_0)\twoheadrightarrow
	(\Scomp_{\neq\emptyset},\circ_k)$ and
	$\overline\wp(\mathsf{Inc}_1)\twoheadrightarrow
	(\Scomp_{\neq\emptyset},\circ_k^c)$, both by downward closure.
\end{remark}

\section{Actions on arrows}
\label{section:polyhedral_calculus}
For an arrow $f\colon Y\to X$, the monotone coordinate maps
$c_0,\id,c_1$ give $\id_Y,f,\id_X$; the remaining map $\tau\in E$
would require a reverse arrow $X\to Y$. In a symmetric monoidal
category, tensor products give the cubical $\mathsf{Inc}$-action up to
coherent natural isomorphism. Gluing these arrows gives the
$\RelScomp_+$-action when tensoring preserves nonempty finite colimits;
preservation of nonempty finite limits gives the dual action.

In this section, in arity $n$, $\underline f$ denotes a tuple of $n$ arrows in the category $\mathrm C$.
When arrows occur as inputs to $\cal z_n$ or
$\cal Z_n$, they are concatenated in the displayed order.
Here, we use the subset inclusion notation $I_0\subseteq I_1$.

\subsection{Cubical action}
\label{section:inclusion_algebra}
Fix parenthesizations of iterated tensor products and use the canonical
associativity isomorphisms. The empty tensor product is $e$.

\begin{definition}[Cubical diagram]\label{definition:cubical_prod_funct}
	Let $(\mathrm C,\otimes,e)$ be a monoidal category and $n\geq0$.
	Fix a sequence of arrows
	$\underline f={(f_i\colon Y_i\to X_i)}_{i\in[n]}$ in $\mathrm C$.
	The \emph{cubical diagram}
	\[
		\cal z_n(-;\underline f)\colon\mathcal B_n\longrightarrow\mathrm C
	\]
	is defined on objects by
	\[
		\cal z_n(I;\underline f)\coloneq\bigotimes_{i=1}^n Z_i^I,\qquad
		Z_i^I\coloneq
		\begin{cases}
			X_i, & i\in I,\\
			Y_i, & i\notin I,
		\end{cases}
	\]
	and on a morphism $I_0\subseteq I_1$ by
	\[
		\cal z_n(I_0\subseteq I_1;\underline f)
		\coloneq\bigotimes_{i=1}^n\phi_i^{I_0\subseteq I_1},\qquad
		\phi_i^{I_0\subseteq I_1}\coloneq
		\begin{cases}
			\id_{X_i}, & i\in I_0,\\
			f_i,        & i\in I_1\setminus I_0,\\
			\id_{Y_i}, & i\notin I_1.
		\end{cases}
	\]
	For $n=0$, this is the one-object diagram with value $e$.
	For an inclusion $I_0\subseteq I_1$ in $\mathcal B_n$, write
	\[
		\cal z_{I_0\subseteq I_1}(\underline f)
		\coloneq\cal z_n(I_0\subseteq I_1;\underline f)
	\]
	for the associated \emph{cubical arrow}. As the input arrows vary,
	this defines the \emph{cubical-arrow functor}
	$\cal z_{I_0\subseteq I_1}\colon\operatorname{Arr}(\mathrm C)^n
	\to\operatorname{Arr}(\mathrm C)$.
\end{definition}
For $I_0\subseteq I_1\subseteq I_2$, the coordinate maps satisfy
$\phi_i^{I_1\subseteq I_2}\circ\phi_i^{I_0\subseteq I_1}
=\phi_i^{I_0\subseteq I_2}$. Tensoring these identities proves
functoriality in $\mathcal B_n$; tensoring commutative squares proves
functoriality in the input arrows.

\begin{proposition}[The cubical action of $\mathsf{Inc}$]
	\label{proposition:inc_cubical_algebra}
	Let $(\mathrm C,\otimes,e)$ be symmetric monoidal. For $m,n\geq1$,
	$(I_0\subseteq I_1)\in\mathsf{Inc}(n)$,
	$(J_0\subseteq J_1)\in\mathsf{Inc}(m)$, and $k\in[n]$, there is a
	canonical natural insertion isomorphism of functors
	$\operatorname{Arr}(\mathrm C)^{n+m-1}
	\to\operatorname{Arr}(\mathrm C)$:
	\[
		\begin{aligned}
		&\cal z_{(I_0\subseteq I_1)\circ_k (J_0\subseteq J_1)}(f_1,\ldots,f_{n+m-1})\\
		&\qquad\cong
		\cal z_{I_0\subseteq I_1}\bigl(f_1,\ldots,f_{k-1},
		\cal z_{J_0\subseteq J_1}(f_k,\ldots,f_{k+m-1}),
		f_{k+m},\ldots,f_{n+m-1}\bigr).
		\end{aligned}
	\]
	Together with the canonical unit and permutation isomorphisms, these
	satisfy the nested-slot, disjoint-slot, unit, and equivariance
	coherence identities. Thus the cubical-arrow functors define an action of
	$\mathsf{Inc}$ on $\operatorname{Arr}(\mathrm C)$ up to coherent natural
	isomorphism.
	Consequently,
	\[
		(I_0\subseteq I_1;[f_1],\ldots,[f_n])
		\longmapsto[\cal z_{I_0\subseteq I_1}(f_1,\ldots,f_n)]
	\]
	makes $\operatorname{IsoArr}(\mathrm C)$ an algebra over the
	set operad $\mathsf{Inc}$.
\end{proposition}
\begin{proof}
	Give the arrow category its pointwise symmetric monoidal structure.
	Inserting $(J_0\subseteq J_1)$ into a slot of type $c_0$, $\id$,
	or $c_1$ replaces that slot by
	$(J_0\subseteq J_0)$, $(J_0\subseteq J_1)$, or
	$(J_1\subseteq J_1)$, respectively. These are precisely the blocks
	in $(I_0\subseteq I_1)\circ_k(J_0\subseteq J_1)$.
	Rebracketing and reordering tensor factors gives the comparison maps;
	symmetric monoidal coherence gives their identities. On
	$\operatorname{IsoArr}(\mathrm C)$ the comparisons become equalities,
	which proves the algebra axioms.
\end{proof}

Under the endpoint suboperads $\mathsf{Inc}_0$ and $\mathsf{Inc}_1$, a
subset $I\subseteq[n]$ acts, respectively, by
\[
	\cal z_{\emptyset\subseteq I}(\underline f),
	\qquad
	\cal z_{I\subseteq[n]}(\underline f).
\]
At $I=\emptyset$ the first arrow is the identity of
$\bigotimes_iY_i$; at $I=[n]$ the second is the identity of
$\bigotimes_iX_i$. Removing these two operations from their respective
endpoint suboperads gives the copy of $\opi{ComTrias}$ indexed by
nonempty subsets and its complementary copy indexed by proper subsets.

\subsection{Polyhedral-colimit action}
\label{section:relative_operad_algebra}
We work with $\RelScomp_+$ so that every indexing face category is
nonempty. With an initial object preserved by tensoring on either side,
the action extends to $\RelScomp$; see the remark after
Theorem~\ref{theorem:algebra_pair}.

\begin{definition}[Face category]\label{def:face_category}
	For a simplicial complex $K$ on $[n]$, its \emph{face category}
	$\operatorname{cat}(K)$ is the full subcategory of $\mathcal B_n$ on
	its faces. We write
	$\operatorname{cat}(K)^{\mathrm{op}}$ for the opposite category.
\end{definition}
Thus $\operatorname{cat}(\emptyset)$ is empty, whereas $\emptyset$ is an
initial face of every nonvoid complex.

Let $\mathrm C$ be a monoidal category with colimits of nonempty finite
diagrams, and choose a colimit functor for each nonempty finite indexing
category.

\begin{definition}[Polyhedral colimits]\label{def:polyhedral_colimit}
	For a nonvoid complex $K$ on $[n]$, put
	\[
		\cal Z_n(K;\underline f)
		\coloneq\operatorname*{colim}_{I\in\operatorname{cat}(K)}
		\cal z_n(I;\underline f).
	\]
	For $(K,L)\in\RelScomp_+(n)$, the inclusion of face categories
	induces the \emph{polyhedral-colimit arrow}
	\[
		\cal Z_n(K,L;\underline f)\colon
		\cal Z_n(L;\underline f)\longrightarrow\cal Z_n(K;\underline f).
	\]
	As the input arrows vary, the chosen colimit functors define the
	\emph{polyhedral-colimit functor}
	\[
		\cal Z_{K,L}\colon\operatorname{Arr}(\mathrm C)^n
		\longrightarrow\operatorname{Arr}(\mathrm C),\qquad
		\underline f\longmapsto\cal Z_n(K,L;\underline f).
	\]
\end{definition}
The construction is functorial under inclusions of complexes and
commutative squares between input arrows. The lower complex supplies
the source of the polyhedral-colimit arrow, and the upper complex its
target.
The identities
\[
	\cal Z_n(\{\emptyset\};\underline f)
	\cong\bigotimes_{i=1}^nY_i,\qquad
	\cal Z_n(\Delta_{[n]};\underline f)
	\cong\bigotimes_{i=1}^nX_i
\]
give the evaluations of the trivial complex and the full simplex.

The face category of a simplex $\Delta_I$ has terminal object $I$,
so there is a canonical isomorphism
\[
	\cal Z_n(\Delta_I;\underline f)\cong\cal z_n(I;\underline f).
\]
Thus $p$ realizes subset inclusions as the corresponding cubical arrows.
The quotient of Section~\ref{section:principal_pair_quotient} can also
be realized through colimits of cubical arrows.
Regard $\mathsf{Inc}(n)=\operatorname{Ob}
\operatorname{Arr}(\mathcal B_n)$ as a poset under componentwise inclusion.
For fixed $\underline f$, cubical functoriality gives a functor
\[
	\mathsf{Inc}(n)\longrightarrow\operatorname{Arr}(\mathrm C),\qquad
	(I_0\subseteq I_1)\longmapsto
	\cal z_{I_0\subseteq I_1}(\underline f).
\]

For a family $\mathcal A\subseteq\mathsf{Inc}(n)$, let
$\mathcal A^\downarrow$ be the down-set it generates in the
componentwise order; that is,
	\[
		\mathcal A^\downarrow\coloneq
		\left\{(H_0\subseteq H_1)\in\mathsf{Inc}(n)\ \middle|\
		H_0\subseteq I_0\text{ and }H_1\subseteq I_1
		\text{ for some }(I_0\subseteq I_1)\in\mathcal A\right\}.
	\]

\begin{lemma}[Gluing cubical arrows]
	\label{lemma:cubical_gluing_quotient}
	For $n\geq1$, a tuple $\underline f$ of $n$ arrows, and a nonempty
	family $\mathcal A\subseteq\mathsf{Inc}(n)$, there is a canonical natural
	isomorphism in
	$\operatorname{Arr}(\mathrm C)$
	\[
		\operatorname*{colim}_{(H_0\subseteq H_1)
		\in\mathcal A^\downarrow}
		\cal z_{H_0\subseteq H_1}(\underline f)
		\cong
		\cal Z_n(q_n(\mathcal A);\underline f).
	\]
	It is natural in commutative squares between the input arrows and
	under inclusions of the indexing family $\mathcal A$.
\end{lemma}
\begin{proof}
	The source and target projections from $\mathcal A^\downarrow$ to
	$\operatorname{cat}(\mathcal A_0^\downarrow)$ and
	$\operatorname{cat}(\mathcal A_1^\downarrow)$ have left adjoints
	$J\mapsto(J\subseteq J)$ and
	$J\mapsto(\emptyset\subseteq J)$, respectively. Both projections
	are therefore final. Since colimits in the arrow category are
	computed at source and target, the displayed colimit is the arrow
	$\cal Z_n(\mathcal A_0^\downarrow;\underline f)\to
	\cal Z_n(\mathcal A_1^\downarrow;\underline f)$, which is
	$\cal Z_n(q_n(\mathcal A);\underline f)$. The comparison uses the
	cubical cocones, so it is natural in the inputs and indexing family.
\end{proof}

\begin{lemma}[Nesting of polyhedral colimits]\label{lemma:face_category_finality}
	Suppose that, for every object $A$ of $\mathrm C$, both
	$A\otimes-$ and $-\otimes A$ preserve colimits of nonempty finite
	diagrams. For $m,n\geq1$, a nonvoid simplicial complex $R$ on $[n]$,
	nonvoid simplicial complexes $N\subseteq M$ on $[m]$, $k\in[n]$, and
	$n+m-1$ input arrows, there is a canonical natural isomorphism
	\begin{equation}\label{equation:inserted_block_colimit}
		\begin{aligned}
			&\cal Z_{n+m-1}(R\triangleleft_k(M,N);f_1,\ldots,f_{n+m-1})\\
		&\qquad \cong
		\cal Z_n\bigl(R;f_1,\ldots,f_{k-1},
		\cal Z_m(M,N;f_k,\ldots,f_{k+m-1}),
		 f_{k+m},\ldots,f_{n+m-1}\bigr).
		\end{aligned}
	\end{equation}
	The isomorphism is natural in commutative squares between the input
	arrows and under inclusions of the outer complex $R$.
\end{lemma}
\begin{proof}
	Let
	\[
		D\colon\operatorname{cat}(R)\longrightarrow\mathsf{Poset},
		\qquad
		D(I)=
		\begin{cases}
			N,&k\notin I,\\
			M,&k\in I,
		\end{cases}
	\]
	be the diagram with the evident identity and inclusion transition maps.
	Let $\mathcal G=\int_{\operatorname{cat}(R)}D$ be its Grothendieck
	construction. Explicitly, its elements are pairs $(I,J)$ with
	$I\in R$, where $J\in N$ if $k\notin I$ and $J\in M$ if $k\in I$,
	ordered componentwise.
	Equation~\eqref{equation:relative_face_action} defines
	\[
		u\colon\mathcal G\longrightarrow
		\operatorname{cat}(R\triangleleft_k(M,N)),\qquad
		(I,J)\longmapsto
		I^{<k}\sqcup(J+(k-1))\sqcup(I^{>k}+(m-1)).
	\]
	Every target face has a unique block decomposition
	\[
		F=S^{<k}\sqcup(T+(k-1))\sqcup(S^{>k}+(m-1)),
		\qquad S\subseteq[n]\setminus\{k\},\quad T\subseteq[m].
	\]
	Put
	\[
		v(F)=\begin{cases}
			(S,T),          & T\in N,    \\
			(S\cup\{k\},T), & T\notin N.
		\end{cases}
	\]
	The face formula and downward closure show that $v(F)\in\mathcal G$.
	If $T\notin N$ and $F\subseteq u(I,J)$, then $k\in I$, since
	otherwise $T\subseteq J\in N$. Hence
	\[
		v(F)\leq(I,J)\quad\Longleftrightarrow\quad F\subseteq u(I,J).
	\]
	Thus $v(F)$ is initial in $(F\mathbin\downarrow u)$ and $u$ is final.
	Tensor preservation and the iterated-colimit (Fubini) isomorphism
	express the right-hand side of~\eqref{equation:inserted_block_colimit}
	as the colimit over $\mathcal G$. Finality identifies this with the
	left-hand side. The comparison uses the original cubical cocones,
	which also give its naturality.
\end{proof}

\begin{theorem}[The action of $\RelScomp_+$]\label{theorem:algebra_pair}
	Let $(\mathrm C,\otimes,e)$ be a symmetric monoidal category with
	colimits of nonempty finite diagrams, preserved by $A\otimes-$ and
	$-\otimes A$ for every object~$A$. For $m,n\geq1$,
	$(K,L)\in\RelScomp_+(n)$, $(M,N)\in\RelScomp_+(m)$, and $k\in[n]$, the
	polyhedral-colimit functors have a canonical natural insertion
	isomorphism of functors $\operatorname{Arr}(\mathrm C)^{n+m-1}
	\to\operatorname{Arr}(\mathrm C)$:
	\[
		\begin{aligned}
		&\cal Z_{(K,L)\circ_k^\ast(M,N)}
		(f_1,\ldots,f_{n+m-1})\\
		&\qquad\cong
		\cal Z_{K,L}\bigl(f_1,\ldots,f_{k-1},
		\cal Z_{M,N}(f_k,\ldots,f_{k+m-1}),
		f_{k+m},\ldots,f_{n+m-1}\bigr).
		\end{aligned}
	\]
	Together with the canonical unit and permutation isomorphisms, these
	satisfy the nested-slot, disjoint-slot, unit, and equivariance
	coherence identities. Thus the polyhedral-colimit functors define an
	action of $\RelScomp_+$ on $\operatorname{Arr}(\mathrm C)$ up to coherent
	natural isomorphism. Consequently,
	\[
		((K,L);[f_1],\ldots,[f_n])
		\longmapsto
		[\cal Z_{K,L}(f_1,\ldots,f_n)]
	\]
	makes $\operatorname{IsoArr}(\mathrm C)$ an algebra over the set
	operad $\RelScomp_+$.
\end{theorem}
\begin{proof}
	Apply Lemma~\ref{lemma:face_category_finality} to the outer
	complexes $L$ and $K$ to obtain the source and target comparison
	isomorphisms. Naturality under $L\subseteq K$ makes these
	comparisons a commutative square, hence an insertion isomorphism in
	the arrow category. Terminal faces identify the relative unit with
	the input arrow, and the symmetry of $\mathrm C$ gives permutations.
	For nested and disjoint insertions, the two composites use the same
	cubical cocones after monoidal coherence. Uniqueness of colimit maps
	proves coherence. On isomorphism classes, these comparisons become
	equalities and give the algebra axioms.
\end{proof}

\begin{remark}
	Under the hypotheses of Theorem~\ref{theorem:algebra_pair}, suppose in
	addition that $\mathrm C$ has a chosen initial object $0_{\mathrm C}$
	preserved by tensoring on either side. Taking
	$\cal Z_n(\emptyset;\underline f)\coloneq 0_{\mathrm C}$ as the empty colimit, the
	same proof extends the action to the full operad $\RelScomp$. In this
	extension, $(K,\emptyset)$ acts by
	$0_{\mathrm C}\to\cal Z_n(K;\underline f)$ and
	$(\emptyset,\emptyset)$ acts by $\id_{0_{\mathrm C}}$.
\end{remark}

For a singleton family $\mathcal A$, its sole subset inclusion is
terminal in $\mathcal A^\downarrow$. Thus $q\eta^\wp=p$ identifies the
gluing operation with the cubical operation of
Proposition~\ref{proposition:inc_cubical_algebra}. On the substitution
and composition endpoint suboperads of $\RelScomp_+$, the action gives,
for $K\in\Scomp_{\neq\emptyset}(n)$, the arrows
\[
	\bigotimes_iY_i\longrightarrow\cal Z_n(K;\underline f),
	\qquad
	\cal Z_n(K;\underline f)\longrightarrow\bigotimes_iX_i.
\]

\subsection{Polyhedral-limit action}
Passing to the opposite category gives a dual construction. Let
$(\mathrm C,\otimes,e)$ be symmetric monoidal.
For arrows $f_i\colon X_i\to Y_i$, apply
Definition~\ref{definition:cubical_prod_funct} in $\mathrm C^{\mathrm{op}}$
to $f_i^{\mathrm{op}}\colon Y_i\to X_i$. Taking opposites gives the
\emph{dual cube}
\[
	\cal z_n^\vee(-;\underline f)
	\coloneq\bigl(\cal z_n^{\mathrm C^{\mathrm{op}}}
	(-;\underline f^{\mathrm{op}})\bigr)^{\mathrm{op}}
	\colon\mathcal B_n^{\mathrm{op}}\longrightarrow\mathrm C.
\]
If $\mathrm C$ has limits of nonempty finite diagrams, choose the
corresponding limit functors and tensor parenthesizations. For
$(K,L)\in\RelScomp_+(n)$, write
\[
	\cal Z^\vee_n(K,L;\underline f)\colon
	\lim_{\operatorname{cat}(K)^{\mathrm{op}}}\cal z_n^\vee
	\longrightarrow
	\lim_{\operatorname{cat}(L)^{\mathrm{op}}}\cal z_n^\vee
\]
for the resulting polyhedral-limit arrow. As the inputs vary, this defines
a functor $\cal Z_{K,L}^\vee\colon\operatorname{Arr}(\mathrm C)^n
\to\operatorname{Arr}(\mathrm C)$.

\begin{corollary}[The dual action of $\RelScomp_+$]
	\label{corollary:dual_polyhedral_limit_action}
	Let $(\mathrm C,\otimes,e)$ be a symmetric monoidal category with
	limits of nonempty finite diagrams, preserved by $A\otimes-$ and
	$-\otimes A$ for every object~$A$. Then the polyhedral-limit
	functors define an action of $\RelScomp_+$ on
	$\operatorname{Arr}(\mathrm C)$ up to coherent natural isomorphism.
	Consequently,
	\[
		((K,L);[f_1],\ldots,[f_n])
		\longmapsto[\cal Z_{K,L}^\vee(\underline f)]
	\]
	makes $\operatorname{IsoArr}(\mathrm C)$ an algebra over the set
	operad $\RelScomp_+$.
\end{corollary}
\begin{proof}
	Limits of nonempty finite diagrams and their preservation in $\mathrm C$
	become the corresponding colimits and their preservation in
	$\mathrm C^{\mathrm{op}}$. Apply
	Theorem~\ref{theorem:algebra_pair} there and transport its action across
	$\operatorname{Arr}(\mathrm C^{\mathrm{op}})
	\cong\operatorname{Arr}(\mathrm C)^{\mathrm{op}}$. Taking opposites
	gives the stated covariant action and its coherence isomorphisms.
\end{proof}

If $\mathrm C$ also has a chosen terminal object $1_{\mathrm C}$ preserved
by tensoring on either side, use $1_{\mathrm C}$ as the limit of the
empty diagram to extend the dual action to $\RelScomp$. Then
$(K,\emptyset)$ acts by
\[
	\lim_{\operatorname{cat}(K)^{\mathrm{op}}}\cal z_n^\vee
	\longrightarrow 1_{\mathrm C},
\]
while $(\emptyset,\emptyset)$ acts by $\id_{1_{\mathrm C}}$.

\section{Applications}\label{section:applications}
The colimit action gives nesting identities for polyhedral and smash
products, and the dual action gives Stanley--Reisner quotient arrows.
Link--deletion leads to PL pairs and loop decompositions.

Write $\mathsf{CGWH}$ for compactly generated weak Hausdorff spaces and
$\mathsf{CGWH}_*$ for based compactly generated weak Hausdorff spaces.
Both categories are cocomplete; $\mathsf{CGWH}$ is cartesian closed,
and $\mathsf{CGWH}_*$ is closed symmetric monoidal under smash product.
Thus Cartesian product and smash product with a fixed object preserve
colimits of nonempty finite diagrams. For Cartesian tensor and inclusions
of CW-subcomplexes
$\underline\iota=(Y_i\hookrightarrow X_i)$, the object
$\cal Z_n(K;\underline\iota)$ is the usual polyhedral product $\cal Z_K$~\cite{Bahri_Bendersky_Cohen_Gitler_2010}.
For nonvoid complexes $L\subseteq K$, Theorem~\ref{theorem:algebra_pair}
gives the arrow $\cal Z_n(K,L;\underline\iota)$ and identifies nesting
with relative join composition. For these CW-pairs, the arrow is the inclusion
\[
	\cal Z_n(L;\underline\iota)\subseteq\cal Z_n(K;\underline\iota)
\]
of unions of product subspaces.
For nonvoid complexes $K$ on $[n]$ and $L$ on $[m]$, and $k\in[n]$,
taking $(R,M,N)=(K,\Delta_{[m]},L)$ in
Lemma~\ref{lemma:face_category_finality} gives
\[
	\begin{aligned}
		\cal Z_{n+m-1}(K\circ_k^cL;\iota_1,\ldots,\iota_{n+m-1})
		&\cong\cal Z_n\Bigl(K;\iota_1,\ldots,\iota_{k-1}, \\
		&\qquad
		\Bigl(\cal Z_m(L;\iota_{k},\ldots,\iota_{k+m-1})
		\longrightarrow\prod_{j=k}^{k+m-1}X_j\Bigr), \\
		&\qquad \iota_{k+m},\ldots,\iota_{n+m-1}\Bigr).
	\end{aligned}
\]
This is the one-slot case of Ayzenberg's nested formula
\cite[Proposition~5.1]{Ayzenberg_2013}; see also
\cite[Theorem~2.9]{Vidaurre_2018} for polyhedral products over
polyhedral joins.

Theorem~\ref{theorem:algebra_pair} also applies to smash products in
$\mathsf{CGWH}_*$.

\subsection{Stanley--Reisner quotient arrows}
The dual polyhedral-limit construction is suited to face rings: joins on
disjoint vertex sets give tensor products, and unions give pullbacks.
Let $\Bbbk$ be a field, and let $\mathsf{CAlg}_{\Bbbk}$ be the category
of unital commutative $\Bbbk$-algebras with tensor product
$\otimes_{\Bbbk}$. For a nonvoid simplicial complex $M$ on a finite
vertex set $E$, its Stanley--Reisner ring and ideal are
\[
	\Bbbk[M]\coloneq\Bbbk[x_v:v\in E]/I_M, \quad
	I_M\coloneq(x_\tau\mid \tau\subseteq E,\ \tau\notin M),
\]
where $x_\tau\coloneq\prod_{v\in \tau}x_v$, with $x_\emptyset=1$. An inclusion
$N\subseteq M$ induces a quotient
$\pi_{M,N}\colon\Bbbk[M]\twoheadrightarrow\Bbbk[N]$.
Write
\[
	\varepsilon\colon\Bbbk[x]\twoheadrightarrow\Bbbk,\qquad x\longmapsto0.
\]
\begin{corollary}[Stanley--Reisner quotient arrows]
	\label{corollary:stanley_reisner_arrows}
	For every $(K,L)\in\RelScomp_+(n)$, there is a canonical
	isomorphism in $\operatorname{Arr}(\mathsf{CAlg}_{\Bbbk})$:
	\[
		\cal Z^\vee_n(K,L;\varepsilon,\ldots,\varepsilon)
		\cong\pi_{K,L}.
	\]
	The isomorphism classes of these quotient arrows form an algebra
	over the relative join operad $\RelScomp_+$.
\end{corollary}
\begin{proof}
	The forgetful functor to vector spaces creates nonempty finite limits.
	Tensoring over $\Bbbk$ preserves them, since it preserves equalizers
	by flatness and commutes with nonempty finite products. Thus
	Corollary~\ref{corollary:dual_polyhedral_limit_action} applies.
	At a face $F$, the dual cube is the polynomial ring
	$\Bbbk[x_i:i\in F]$, with restriction maps setting variables outside
	$F$ to zero. For any nonvoid complex $S$ on $[n]$, these maps identify
	$\Bbbk[S]$ with the limit of the dual cube over
	$\operatorname{cat}(S)^{\mathrm{op}}$: a compatible family has exactly
	the coefficients of monomials whose supports are faces of $S$.
	Naturality under $L\subseteq K$ identifies the induced arrow with
	$\pi_{K,L}$. Thus the quotient-arrow classes are exactly the
	evaluations of $\RelScomp_+$ on copies of $[\varepsilon]$, which form
	a subalgebra by the dual action.
\end{proof}

For $(K,L)\in\RelScomp_+$, the kernel of
$\Bbbk[K]\twoheadrightarrow\Bbbk[L]$ is $I_L/I_K$, the relative
Stanley--Reisner $\Bbbk[K]$-module~\cite{Adiprasito_Sanyal_2016}.
The quotient arrow therefore records information beyond $\Bbbk[K]$.

\subsection{PL ball--boundary pairs}\label{section:PL_subop}
For a PL~sphere $K$, the real moment-angle complex $\mathbb R\cal Z_K$ is a
PL~manifold, and the usual moment-angle complex $\cal Z_K$ is a
topological manifold~\cite[Theorem~2.3 and Corollary~2.10]{Cai_2017}.
For earlier work on
moment-angle manifolds over polytopal spheres, see
\cite{Buchstaber_Panov_1998,Buchstaber_Panov_1999,Buchstaber_Panov_2015}.
Thus PL~spheres provide a natural class of complexes in toric topology.
Suspension by $S^0=\partial\Delta_{[2]}$ preserves PL~spheres and
PL~balls. The simplicial wedge is PL homeomorphic to a suspension and
has the same property~\cite[Lemma~6 and Corollary~2(a)]{Bagchi_Datta_1998}.
Both operations arise from the relative insertion of the PL
ball--boundary pair $(\Delta_{[2]},\partial\Delta_{[2]})$: by the
boundary specialization in Section~\ref{section:link_deletion}, a ghost
slot gives suspension, while a non-ghost slot gives the simplicial
wedge. Insertion acts on each component of a PL ball--boundary pair.
We show that this closure extends to insertion of arbitrary PL
ball--boundary pairs, which form a suboperad of $\RelScomp_+$.
The proof uses standard PL join and link formulas
\cite{Hudson_1969,Rourke_Sanderson_1982}, and we also identify a
vertex-spanning suboperad.

A PL~sphere is a finite simplicial complex whose realization is PL
homeomorphic to a sphere, and a PL~ball is defined analogously. We also
use the augmented convention that $\{\emptyset\}$, whose realization is
empty, is the $(-1)$-sphere~$S^{-1}$; in particular,
$\partial\Delta_{[1]}=\{\emptyset\}$.

Joining a PL ball $B$ with a sphere $S$ gives a ball
with boundary $(\partial B)\ast S$
\cite[Proposition~2.23]{Rourke_Sanderson_1982}; for $S^{-1}$ this is
the identity $B\ast\{\emptyset\}=B$. In this paper, a
\emph{PL ball--boundary pair} is a relative complex $(B,\partial B)$
in which $B$ is a finite PL $d$-ball for some $d\geq0$ and the lower
complex is its boundary, a PL
$(d-1)$-sphere; interior vertices of $B$ and ambient ghost vertices are
allowed. Thus the operadic unit $(\Delta_{[1]},\{\emptyset\})$ is the
pair consisting of a $0$-ball and its boundary. Write $\mathsf{PLpairs}$
for this collection.

\begin{proposition}\label{proposition:PL_pair_subop}
	The collection $\mathsf{PLpairs}$ of PL ball--boundary pairs is a
	unital suboperad of $\RelScomp_+$. More precisely, for any two
	PL ball--boundary pairs, after restricting to their non-ghost vertex
	sets and forgetting the prescribed ambient labels, the composite
	$(B,\partial B)\circ_k^\ast(B',\partial B')$ is PL homeomorphic, as a
	pair, to $\bigl(B\ast\partial B',\,\partial B\ast\partial B'\bigr)$.
\end{proposition}

\begin{proof}
	Relabeling preserves PL ball--boundary pairs.
	Put $(D,E)=(B,\partial B)\circ_k^\ast(B',\partial B')$. If $k$ is
	ghost, the link--deletion formula gives the asserted join pair after
	restricting to the non-ghost vertex sets and forgetting the ambient
	labels. Suppose $k$ is non-ghost. The pushout decomposition of $D$
	glues $\operatorname{Lk}_k(B)\ast B'$ to
	$(B\setminus k)\ast\partial B'$ along their intersection
	$\operatorname{Lk}_k(B)\ast \partial B'$. The extension lemma
	\cite[Lemma~1.10]{Rourke_Sanderson_1982} identifies $B'$ with the cone
	on $\partial B'$ relative to its boundary; the dimension-zero case
	is immediate. Undo the block relabeling and identify the cone vertex
	with $k$. Regard $\Delta_{\{k\}}$ as a complex on the singleton block
	$\{k\}$ and the link as a complex on the complementary block. The
	induced map on
	$\operatorname{Lk}_k(B)\ast B'$ agrees with the identity on the common
	subcomplex $\operatorname{Lk}_k(B)\ast\partial B'$. These maps glue
	to a homeomorphism with $B\ast\partial B'$, since
	\[
		B=(\Delta_{\{k\}}\ast\operatorname{Lk}_k(B))\cup(B\setminus k).
	\]
	Applying the same argument to $\partial B$ shows that $E$ is PL
	homeomorphic to
	$\partial B\ast\partial B'$ when $k$ is a boundary vertex; for an
	interior vertex, $k$ is ghost in $\partial B$, so the ghost case gives
	the same conclusion. The cone extension is relative to $\partial B'$,
	so these homeomorphisms fit together as the asserted PL homeomorphism
	of pairs. The join formulas show that $D$ is a PL ball with boundary
	$E$.
\end{proof}

A non-unit PL ball--boundary pair $(B,\partial B)$ is \emph{vertex-spanning}
if every non-ghost vertex of $B$ belongs to $\partial B$. Equivalently,
$\partial B$ is a spanning subcomplex of $B$, or $B$ and $\partial B$
have the same ghost vertices. Let
$\mathsf{PLpairs}^{\mathrm{vsp}}$ consist of these pairs together with the
operadic unit, adjoined by convention; here $\mathrm{vsp}$ abbreviates
``vertex-spanning''.

\begin{proposition}
	The collection $\mathsf{PLpairs}^{\mathrm{vsp}}$ is a unital suboperad of $\mathsf{PLpairs}$.
\end{proposition}
\begin{proof}
	The unit law handles either factor being the unit. For two non-unit
	vertex-spanning pairs, every surviving outer vertex lies in the boundary
	of the composite. Every inserted vertex does too when the substituted
	slot is non-ghost. In a ghost slot both components are joined to $\partial B'$.
\end{proof}

\subsection{Loop decompositions of arrows}
\label{section:relative_Eldridge}

Eldridge's loop decomposition~\cite[Theorem~4.3]{Eldridge_2026} can be
chosen compatibly for the lower and upper complexes. This gives an
equivalence of arrows and a splitting of the looped relative homotopy
fiber.

All spaces below are based objects of $\mathsf{CGWH}$, and homotopy
pushouts are formed on their underlying spaces. For an arrow $f$, write
$\Omega f$ for its looped arrow and
$\mathbf1_Y\coloneq(Y\xrightarrow{\id}Y)$ for an identity
arrow. For based spaces $U,V$, their ordinary join $U\ast V$ is the
double mapping cylinder of $U\leftarrow U\times V\to V$, based at the
midpoint over the basepoints of $U$ and $V$. Products and joins of
arrows are formed objectwise. We work in the
homotopy category of arrows obtained by inverting objectwise homotopy
equivalences, and use functorial based mapping-path homotopy fibers,
denoted by $\operatorname{hofib}$. If
$u=(u_0,u_1)\colon f\to g$ is a morphism of arrows, then
$\operatorname{Fib}_{\mathrm{Arr}}(u)$ denotes its objectwise
homotopy-fiber arrow
\[
	\operatorname{hofib}(u_0)\longrightarrow\operatorname{hofib}(u_1).
\]
Thus $\operatorname{hofib}(f)$ is a space, while
$\operatorname{Fib}_{\mathrm{Arr}}(u)$ is an arrow. A homotopy
fibration of arrows $F\to E\xrightarrow{p}B$ includes a specified
equivalence $F\simeq\operatorname{Fib}_{\mathrm{Arr}}(p)$ that identifies
the displayed map $F\to E$ with the canonical fiber inclusion in the
arrow homotopy category.

Fix $(K,L)\in\RelScomp_+(n)$, $(M,N)\in\RelScomp_+(m)$, and a non-ghost
vertex $k$ of $L$. Put
\[
	(Q,R)=(K,L)\circ_k^\ast(M,N).
\]
Fix pointed CW-pairs
$\underline\iota=(A_j\hookrightarrow X_j)_{j\in[n+m-1]}$, with each
$A_j$ a based CW-subcomplex of $X_j$. Split the inputs into inserted
and outside tuples:
\[
	\underline\iota^{\mathrm{in}}
	=(\iota_{k},\iota_{k+1},\ldots,\iota_{k+m-1}),
	\qquad
	\underline\iota^{\mathrm{out}}
	=(\iota_1,\ldots,\iota_{k-1},
	  \iota_{k+m},\ldots,\iota_{n+m-1}).
\]

Write
\[
	\operatorname{Lk}_k(K,L)
	\coloneq(\operatorname{Lk}_k(K),\operatorname{Lk}_k(L)),
	\qquad
	(K,L)\setminus k\coloneq(K\setminus k,L\setminus k).
\]
In the polyhedral products below, evaluate each link or deletion at
$k$ on its $n-1$ outside vertices: discard the ghost coordinate $k$
retained in the combinatorial formulas, relabel in order, and use
$\underline\iota^{\mathrm{out}}$.
For $n=1$, the vertex condition forces $K=L=\Delta_{[1]}$. We use
$\cal Z_0(\{\emptyset\};())=*$ and
$\cal Z_0(\{\emptyset\},\{\emptyset\};())=\mathbf1_*$, so all outside
polyhedral-product arrows are identities of a point. The link
inclusions give a morphism of arrows
\[
	\lambda_k\colon
	\cal Z_{n-1}(\operatorname{Lk}_k(K,L);\underline\iota^{\mathrm{out}})
	\longrightarrow
	\cal Z_{n-1}((K,L)\setminus k;\underline\iota^{\mathrm{out}}).
\]
Set
\[
		\Gamma_k=\operatorname{Fib}_{\mathrm{Arr}}(\lambda_k)
		=(G_L\longrightarrow G_K),\quad
		G_T=\operatorname{hofib}\bigl(
		\cal Z_{n-1}(T\setminus k,\operatorname{Lk}_k(T);
		\underline\iota^{\mathrm{out}})\bigr),
		\qquad T\in\{L,K\}.
\]
On the inserted block, put
\[
	h=\cal Z_m(M,N;\underline\iota^{\mathrm{in}})
	\colon\cal Z_m(N;\underline\iota^{\mathrm{in}})\longrightarrow
	\cal Z_m(M;\underline\iota^{\mathrm{in}}),
	\qquad H=\operatorname{hofib}(h).
\]
The condition that $k$ is a vertex of $L$ is essential. Insertion into
the unit of the relative join operad, where $k$ is ghost in the lower
complex, returns the arbitrary inner arrow $h$; its loop need not admit
the splitting below.

\begin{theorem}[Eldridge decomposition for arrows]
	\label{theorem:Eldridge_arrow_decomposition}
	There is a natural homotopy fibration of arrows
	\begin{equation}\label{equation:Eldridge_arrow_fibration}
		\Gamma_k\ast\mathbf1_H
		\longrightarrow\cal Z_{n+m-1}(Q,R;\underline\iota)
		\xrightarrow{\ \Theta\ }
		\cal Z_{n-1}((K,L)\setminus k;\underline\iota^{\mathrm{out}})\times
		\mathbf1_{\cal Z_m(M;\underline\iota^{\mathrm{in}})}.
	\end{equation}
	After looping, this fibration splits naturally. Hence
	\begin{equation}\label{equation:Eldridge_arrow_splitting}
		\Omega\cal Z_{n+m-1}(Q,R;\underline\iota)
		\simeq
		\Omega\cal Z_{n-1}((K,L)\setminus k;\underline\iota^{\mathrm{out}})
		\times\mathbf1_{\Omega\cal Z_m(M;\underline\iota^{\mathrm{in}})}
		\times\Omega(\Gamma_k\ast\mathbf1_H).
	\end{equation}
	In the arrow homotopy category, the fibration and splitting are natural
	in based maps of the input CW-pairs and in inclusions of $(K,L)$ and
	$(M,N)$ of fixed arities for which $k$ remains a vertex of the lower
	outer complex.
\end{theorem}

\begin{proof}
	Apply the link--deletion union and intersection formulas to both $L$
	and $K$. Since links, deletions, and one-slot polyhedral joins are
	natural under $L\subseteq K$, the two pushout squares form opposite
	faces of a commutative cube. Regarding the vertical edges as objects
	gives a pushout square in the arrow category of simplicial complexes.

	Polyhedral products take these unions and intersections to unions and
	intersections of product subspaces. The finite product
	$\prod_jX_j$ in $\mathsf{CGWH}$ has its product CW structure
	\cite[Theorem~A.6]{Hatcher_2002}. Each face block is a CW subcomplex,
	and each polyhedral product is a finite union of these subcomplexes.
	Thus the inclusions in the resulting squares are cofibrations, so the
	pushouts are objectwise homotopy pushouts. The natural product formula
	for joins identifies the square of arrows as
	\begin{equation}\label{equation:Eldridge_pushout_of_arrows}
	\begin{tikzcd}[column sep=large,row sep=large]
		\cal Z_{n-1}(\operatorname{Lk}_k(K,L);\underline\iota^{\mathrm{out}})
		\times\mathbf1_{\cal Z_m(N;\underline\iota^{\mathrm{in}})}
		\arrow[r]\arrow[d]
		\arrow[dr,phantom,very near start,"\ulcorner"] &
		\cal Z_{n-1}(\operatorname{Lk}_k(K,L);\underline\iota^{\mathrm{out}})
		\times\mathbf1_{\cal Z_m(M;\underline\iota^{\mathrm{in}})}
		\arrow[d] \\
		\cal Z_{n-1}((K,L)\setminus k;\underline\iota^{\mathrm{out}})
		\times\mathbf1_{\cal Z_m(N;\underline\iota^{\mathrm{in}})}
		\arrow[r] &
		\cal Z_{n+m-1}(Q,R;\underline\iota).
	\end{tikzcd}
	\end{equation}
	The top and left morphisms are induced by $\id\times h$ and
	$\lambda_k\times\id$, respectively. The map of identity arrows
	induced by $h$ has objectwise homotopy fiber $\mathbf1_H$. Map every
	corner to the base of~\eqref{equation:Eldridge_arrow_fibration}.
	The lower-right map is $\Theta$. Its source and target maps are induced
	by
	$T\triangleleft_k(M,N)\subseteq(T\setminus k)\ast_kM$ for
	$T=L$ and $T=K$, respectively.

	Take functorial mapping-path fibers in the arrow category. The four
	side faces of the resulting cube are homotopy pullbacks. The cube
	lemma therefore makes its fiber face a homotopy pushout
	\cite[Lemma~2.18]{Eldridge_2026}. The fibers at the upper-left,
	upper-right, lower-left, and lower-right corners are naturally equivalent
	to $\Gamma_k\times\mathbf1_H$, $\Gamma_k$, $\mathbf1_H$, and
	$\operatorname{Fib}_{\mathrm{Arr}}(\Theta)$, respectively. Projecting away
	the contractible identity fibers and using their natural contraction
	homotopies gives an equivalent homotopy-pushout diagram
	\[
	\begin{tikzcd}[column sep=large,row sep=large]
		\Gamma_k\times\mathbf1_H\arrow[r]\arrow[d]
		\arrow[dr,phantom,very near start,"\ulcorner"] &
		\Gamma_k\arrow[d] \\
		\mathbf1_H\arrow[r] &
		\operatorname{Fib}_{\mathrm{Arr}}(\Theta).
	\end{tikzcd}
	\]
	The two maps out of $\Gamma_k\times\mathbf1_H$ are the projections.
	This is Eldridge's fiber-cube argument~\cite[Lemma~4.2]{Eldridge_2026},
	applied naturally to both endpoints. By the join convention above,
	the homotopy pushout of the projections is the join objectwise, so
	\[
		\operatorname{Fib}_{\mathrm{Arr}}(\Theta)\simeq
		\Gamma_k\ast\mathbf1_H.
	\]
	The homotopy-pushout comparison gives an equivalence
	$\Gamma_k\ast\mathbf1_H\to\operatorname{Fib}_{\mathrm{Arr}}(\Theta)$
	compatible with the maps to the lower-right corner
	of~\eqref{equation:Eldridge_pushout_of_arrows}. Composing it with the
	canonical fiber inclusion gives the first map
	in~\eqref{equation:Eldridge_arrow_fibration}.
	Under this equivalence, the fiber arrow is
	$G_L\ast H\to G_K\ast H$.

	It remains to split the fibration. Since $k$ is a vertex of $L$, both
	link polyhedral products are nonempty and based. Since $N$ is
	nonvoid, its polyhedral product is nonempty as well. The structure maps
	into the pushout and the basepoints in the omitted factors define a
	morphism of arrows
	\[
		\cal Z_{n-1}((K,L)\setminus k;\underline\iota^{\mathrm{out}})
		\vee\mathbf1_{\cal Z_m(M;\underline\iota^{\mathrm{in}})}
		\longrightarrow\cal Z_{n+m-1}(Q,R;\underline\iota).
	\]
	The first wedge summand enters the deletion--$N$ corner of
	\eqref{equation:Eldridge_pushout_of_arrows} using the basepoint of
	$\cal Z_m(N;\underline\iota^{\mathrm{in}})$. The second enters the
	link--$M$ corner using the compatible basepoints of the two link
	polyhedral products. Composing this wedge map with $\Theta$ gives the
	canonical wedge-to-product morphism. After looping, this morphism has
	a natural right homotopy inverse: include the two loop factors and
	multiply loops in that order. Composing this right inverse with the
	loop of the wedge map gives
	\[
		s\colon
		\Omega\cal Z_{n-1}((K,L)\setminus k;\underline\iota^{\mathrm{out}})
		\times\mathbf1_{\Omega\cal Z_m(M;\underline\iota^{\mathrm{in}})}
		\longrightarrow\Omega\cal Z_{n+m-1}(Q,R;\underline\iota),
		\qquad (\Omega\Theta)\circ s\simeq\id.
	\]
	The standard product equivalence for a split looped fibration is
	obtained by multiplying this section by the looped fiber inclusion.
	It gives~\eqref{equation:Eldridge_arrow_splitting}. Mapping-path
	fibers, the contractions of identity fibers, and loop multiplication
	are functorial, which proves the stated naturality.
\end{proof}

The relative consequence is a decomposition of the looped homotopy
fiber of the polyhedral-product arrow:
\[
	\begin{aligned}
	\Omega\operatorname{hofib}
		\bigl(\cal Z_{n+m-1}(Q,R;\underline\iota)\bigr)
		&\simeq
		\Omega\operatorname{hofib}
		\bigl(\cal Z_{n-1}((K,L)\setminus k;
		\underline\iota^{\mathrm{out}})\bigr)\\
		&\quad\times
		\Omega\operatorname{hofib}(G_L\ast H\longrightarrow G_K\ast H).
	\end{aligned}
\]
This follows by taking homotopy fibers
in~\eqref{equation:Eldridge_arrow_splitting}: homotopy fibers preserve
products and commute with looping, while the identity-arrow factor has
contractible homotopy fiber.

At source and target, the theorem recovers Eldridge's one-slot
fibration and splitting~\cite[Theorem~4.3]{Eldridge_2026}. For a diagonal
outer pair $(T,T)$ with no ghost vertices, right-to-left insertion and
Lemma~\ref{lemma:face_category_finality} identify the nested polyhedral
products with the polyhedral product over the classical polyhedral
join. Iterating the splitting in
Eldridge's deletion order recovers her multi-input
formula~\cite[Theorem~1.1]{Eldridge_2026}. Our fixed ambient convention
retains any remaining ghost-coordinate $A_i$-factors.

\subsection{Principal PL pairs}
\label{section:principal_PL_pairs}

Specialize the input inclusions to $S^1\hookrightarrow D^2$. For a
nonvoid complex $T$ on $[n]$, write
\[
	\cal Z_T=\cal Z_n\bigl(T;(S^1\hookrightarrow D^2)_{i\in[n]}\bigr).
\]
For $(K,L)\in\RelScomp_+$, the \emph{moment-angle arrow} is
$\mathfrak Z_n(K,L)\coloneq\cal Z_L\to\cal Z_K$. Define its
\emph{moment-angle homotopy fiber} by
\[
	\mathcal F(K,L)\coloneq\operatorname{hofib}(\mathfrak Z_n(K,L)),
\]
based at $(1,\ldots,1)$.
We identify a PL suboperad for which this fiber is explicit.

For a PL ball--boundary pair $(B,\partial B)$, let
\[
	\operatorname{MinInt}(B)
	\coloneq\min_{\subseteq}(B\setminus\partial B)
\]
be its family of inclusion-minimal interior faces. This family is
nonempty because every facet of $B$ is interior; its members are
nonempty because $\emptyset\in\partial B$. It therefore defines an
element of $\overline\wp(\opi{ComTrias})$. Call the pair
\emph{principal} when the family is a singleton, and write
$\mathsf{PrinPL}$ for the resulting symmetric collection. The unit is
principal, with minimal interior face $\{1\}$.

The terminology agrees with relative Stanley--Reisner theory:
\[
	\ker\bigl(\Bbbk[B]\twoheadrightarrow\Bbbk[\partial B]\bigr)
	=I_{\partial B}/I_B
	=\bigl(x_\tau\mid\tau\in\operatorname{MinInt}(B)\bigr)
	\subseteq\Bbbk[B].
\]
The displayed monomials are the minimal multigraded generators of this
kernel ideal. Thus the pair is principal exactly when the ideal is principal.

\begin{proposition}[Minimal interior faces as an operad morphism]
	\label{proposition:minimal_interior_composition}
	Sending $(B,\partial B)$ to $\operatorname{MinInt}(B)$ defines an
	operad morphism
	\[
		\operatorname{MinInt}\colon
		\mathsf{PLpairs}\longrightarrow
		\overline\wp(\opi{ComTrias}).
	\]
	Consequently, $\mathsf{PrinPL}$ is a unital suboperad of
	$\mathsf{PLpairs}$, and taking the unique minimal interior face gives
	an operad morphism
	\[
		\mathsf{PrinPL}\longrightarrow\opi{ComTrias},
		\qquad (B,\partial B)\longmapsto\sigma_B.
	\]
\end{proposition}

\begin{proof}
	Put $(D,E)=(B,\partial B)\circ_k^\ast(B',\partial B')$, where the
	second pair has arity $m$. Every output subset has a unique form
	\[
		\rho_{k,m}(A,J)
		=A^{<k}\sqcup(J+(k-1))\sqcup(A^{>k}+(m-1)),
	\]
	with $A\subseteq[n]\setminus\{k\}$ and $J\subseteq[m]$.
	Equation~\eqref{equation:relative_face_action} gives
	\[
		\rho_{k,m}(A,J)\in D\setminus E
		\quad\Longleftrightarrow\quad
		\begin{cases}
			J\in\partial B'\text{ and }A\in B\setminus\partial B,\quad\text{or}\\
			J\in B'\setminus\partial B'\text{ and }
			A\cup\{k\}\in B\setminus\partial B
		\end{cases}.
	\]
	In the first case, the output is minimal precisely when $J=\emptyset$
	and $A\in\operatorname{MinInt}(B)$. In the second, it is minimal
	precisely when $A\cup\{k\}\in\operatorname{MinInt}(B)$ and
	$J\in\operatorname{MinInt}(B')$. Therefore
	\begin{align*}
		\operatorname{MinInt}(D)
		={}&\bigl\{
		\sigma^{<k}\sqcup(\sigma^{>k}+(m-1))
		\mid \sigma\in\operatorname{MinInt}(B),\ k\notin\sigma
		\bigr\}\\
		&\mathbin\cup\bigl\{
		\sigma^{<k}\sqcup(\tau+(k-1))\sqcup(\sigma^{>k}+(m-1))
		\mathrel{\big|}
		\substack{\sigma\in\operatorname{MinInt}(B),\ k\in\sigma,\\
			\tau\in\operatorname{MinInt}(B')}
		\bigr\}\\
		={}&\operatorname{MinInt}(B)\circ_k^\wp
		\operatorname{MinInt}(B'),
	\end{align*}
	The last equality is Equation~\eqref{equation:comtrias_composition},
	which proves the first claim, including equivariance and the unit.
	Singleton families form the image of the monad unit
	$\eta^\wp_{\opi{ComTrias}}$ and hence a suboperad. Taking their unique
	elements proves the remaining claims.
\end{proof}

For a face $\sigma$ of a complex $B$ on $[n]$, write
\[
	\operatorname{Lk}_B(\sigma)
	\coloneq\{J\subseteq[n]\setminus\sigma\mid J\cup\sigma\in B\}.
\]
We regard the link on the complementary ambient vertex set, including
ghost vertices. In the join decomposition below, $\Delta_\sigma$ and
$\partial\Delta_\sigma$ use the block $\sigma$, while the link uses the
complementary block.

\begin{lemma}[Structure of principal PL ball--boundary pairs]
	\label{lemma:principal_PL_structure}
	Let $(B,\partial B)\in\mathsf{PrinPL}$, and put
	$\sigma=\sigma_B$ and $L=\operatorname{Lk}_B(\sigma)$.
	Then $L$ is a PL~sphere and
	\[
		(B,\partial B)
		=(\Delta_\sigma\ast L,\,\partial\Delta_\sigma\ast L).
	\]
	Conversely, this join pair is principal, with $\sigma_B=\sigma$,
	whenever $\sigma$ is nonempty and $L$ is a PL~sphere on the
	complementary ambient set. A non-unit principal pair is
	vertex-spanning precisely when $|\sigma_B|\geq2$.
\end{lemma}

\begin{proof}
	The minimal interior face $\sigma$ is nonempty. Every
	interior face contains $\sigma$, and no boundary face can contain it.
	Every facet of $B$ is interior and therefore contains $\sigma$, giving
	$B=\Delta_\sigma\ast L$. The faces not containing $\sigma$ are
	precisely the boundary faces, giving
	$\partial B=\partial\Delta_\sigma\ast L$. The link of an interior
	face of a PL~ball is a PL~sphere
	\cite{Hudson_1969,Rourke_Sanderson_1982}. Conversely, the standard join
	formulas show that $\Delta_\sigma\ast L$ is a PL~ball with boundary
	$\partial\Delta_\sigma\ast L$, and its unique minimal interior face is
	$\sigma$.

	An interior vertex $v$ makes $\{v\}$ a minimal interior face. Hence
	a principal pair has an interior vertex exactly when $|\sigma|=1$,
	with $v$ the unique vertex of $\sigma$. A non-unit principal pair is
	therefore vertex-spanning exactly when $|\sigma|\geq2$.
\end{proof}

In the facet case, $L=\{\emptyset\}=S^{-1}$ is the augmented sphere.
When $\sigma_B=\{v\}$, the decomposition gives the cone pair
$(\Delta_{\{v\}}\ast L,L)$, with $v$ its unique interior vertex. Thus
$\mathsf{PrinPL}^{\mathrm{vsp}}\coloneq\mathsf{PrinPL}\cap
\mathsf{PLpairs}^{\mathrm{vsp}}$ consists of the unit and the non-unit
principal pairs with $|\sigma_B|\geq2$.

\begin{corollary}[Fiber of a principal pair]
	\label{proposition:principal_PL_fiber}
	Let $(B,\partial B)\in\mathsf{PrinPL}$ and put
	$r=|\sigma_B|$. Then
	\[
		\mathcal F(B,\partial B)\simeq S^{2r-1}.
	\]
\end{corollary}

\begin{proof}
	Let $L=\operatorname{Lk}_B(\sigma_B)$. Since
	$\cal Z_{\Delta_{\sigma_B}}=D^{2r}$ and
	$\cal Z_{\partial\Delta_{\sigma_B}}=S^{2r-1}$,
	Lemma~\ref{lemma:principal_PL_structure} and the product formula for
	moment-angle complexes of joins
	\cite[Proposition~7.6]{Panov_2013} give
	\[
		\mathfrak Z_n(B,\partial B)
		\cong
		\bigl(S^{2r-1}\hookrightarrow D^{2r}\bigr)
		\times\id_{\cal Z_L}.
	\]
	Since $D^{2r}$ is contractible, the displayed arrow has homotopy
	fiber $S^{2r-1}$. Common ghost coordinates occur in the same
	$\cal Z_L$-factor at source and target.
\end{proof}

Following Eldridge~\cite[Section~2.1]{Eldridge_2026}, let $\cal W$
consist of spaces homotopy equivalent to finite-type wedges of simply
connected spheres. Let $\cal P$ consist of $H$-spaces homotopy
equivalent to finite-type products of spheres and loop spaces on simply
connected spheres. The corollary gives
$\Sigma\mathcal F(B,\partial B)\simeq S^{2r}\in\cal W$.

Define the \emph{upper $\cal P$-locus} and its vertex-spanning subcollection
\[
	\mathsf{PrinPL}_{\cal P}
	\coloneq\bigl\{(B,\partial B)\in\mathsf{PrinPL}
	\mid\Omega\cal Z_B\in\cal P\bigr\},
	\qquad
	\mathsf{PrinPL}_{\cal P}^{\mathrm{vsp}}
	\coloneq\mathsf{PrinPL}_{\cal P}\cap\mathsf{PrinPL}^{\mathrm{vsp}}.
\]

The spherical fiber calculation makes Eldridge's closure theorem applicable.

\begin{corollary}[Upper $\cal P$-locus]
	\label{corollary:principal_loop_closure}
	Both $\mathsf{PrinPL}_{\cal P}$ and
	$\mathsf{PrinPL}_{\cal P}^{\mathrm{vsp}}$ are unital suboperads. For a
	non-unit $(B,\partial B)\in\mathsf{PrinPL}_{\cal P}$, the source loop
	space $\Omega\cal Z_{\partial B}$ belongs to $\cal P$ if and only if
	$|\sigma_B|\geq2$, equivalently, the pair is vertex-spanning.
\end{corollary}

\begin{proof}
	If $\Omega\cal Z_B\in\cal P$, then $B$ has no ghost vertices:
	each would contribute an $S^1$-factor and hence infinitely many
	components after looping. For two non-unit elements, the preceding
	corollary gives
	$\Sigma\mathcal F(B',\partial B')\simeq
	S^{2|\sigma_{B'}|}\in\cal W$. Since $\Sigma S^1=S^2\in\cal W$,
	Eldridge's theorem~\cite[Theorem~6.1]{Eldridge_2026} puts the upper
	loop space of their composite in $\cal P$. The composite is principal
	by Proposition~\ref{proposition:minimal_interior_composition}; the
	unit and relabeling preserve the condition. Intersecting with
	$\mathsf{PLpairs}^{\mathrm{vsp}}$ gives the second suboperad.

	For a non-unit pair, put $r=|\sigma_B|$. The structure lemma gives
	\[
		\Omega\cal Z_{\partial B}
		\simeq\Omega S^{2r-1}\times\Omega\cal Z_B.
	\]
	The source lies in $\cal P$ exactly when $r\geq2$: for $r=1$,
	$\Omega S^1\simeq\mathbb Z$ excludes it. By
	Lemma~\ref{lemma:principal_PL_structure}, $r\geq2$ is precisely the
	vertex-spanning condition.
\end{proof}

The non-unit qualification is necessary: the operadic unit belongs to
$\mathsf{PrinPL}_{\cal P}^{\mathrm{vsp}}$ by convention, but its source
loop space is $\Omega S^1\simeq\mathbb Z\notin\cal P$.

\section*{Acknowledgments}
I am grateful to my advisor Bruno Vallette for introducing me to operad
theory, for his constant support and invaluable comments, and for many
fruitful discussions about this work. I also thank Anthony Bahri for
motivating conversations about polyhedral products that helped shape
the paper.

\end{document}